\documentclass[final,onefignum,onetabnum]{siamart171218}

\usepackage{lipsum}
\usepackage{amsfonts}
\usepackage{amsmath,amssymb,mathrsfs}
\usepackage{graphicx}
\usepackage{epstopdf}
\usepackage{subfigure}
\usepackage{caption}
\usepackage{enumitem}

\usepackage{algorithm,algpseudocode}
\usepackage{amsbsy,amscd}

\ifpdf
  \DeclareGraphicsExtensions{.eps,.pdf,.png,.jpg}
\else
  \DeclareGraphicsExtensions{.eps}
\fi

\newsiamremark{remark}{Remark}
\newsiamremark{hypothesis}{Hypothesis}
\crefname{hypothesis}{Hypothesis}{Hypotheses}
\newsiamthm{claim}{Claim}

\headers{Bridging and improving inverse solutions with data completion}{Tan Bui-Thanh and Qin Li and Leonardo Zepeda-N\'u\~nez}

\title{Bridging and Improving Theoretical and Computational Electric Impedance Tomography via Data Completion\thanks{Submitted to the editors DATE.
\funding{T.B.T. was partially funded by the National Science Foundation awards NSF-1808576 and NSF-CAREER-1845799; by the Defense Thread Reduction Agency award DTRA-M1802962; by the Department of Energy award DE-SC0018147; by KAUST; by 2018 ConTex award; and by 2018 UT-Portugal CoLab award. The work of Q.L. is supported in part by UW-Madison Data Initiative, Vilas Young Investigation Award and National Science Foundation under the grant DMS-1750488. The work of L.Z.-N. is supported in part by the National Science Foundation under the grant DMS-2012292. In addition, Q.L. and L.Z.-N. are supported by NSF TRIPODS award 1740707. The views expressed in the article do not necessarily represent the views of the any funding agencies. The authors are grateful to the supports.}}}

\author{
Tan Bui-Thanh\thanks{Department of Aerospace Engineering and
  Engineering Mechanics, The Oden Institute for Computational
  Engineering and Sciences, UT Austin,  Austin, Texas
  (\email{tanbui@oden.utexas.edu},
  \url{https://users.oden.utexas.edu/\string~tanbui/}).}
  \and Qin Li\thanks{Mathematics Department and Wisconsin Institute for Discovery, UW-Madison, Madison, WI
  (\email{qinli@math.wisc.edu}, \url{http://www.math.wisc.edu/\string~qinli/}).}
  \and  Leonardo Zepeda-N\'u\~nez\thanks{Mathematics Department, UW-Madison, Madison, WI (\email{lzepeda@math.wisc.edu}, \url{https://www.math.wisc.edu/\string~lzepeda/}).}
}
\usepackage{amsopn}

\makeatletter
\newcommand*{\addFileDependency}[1]{
  \typeout{(#1)}
  \@addtofilelist{#1}
  \IfFileExists{#1}{}{\typeout{No file #1.}}
}
\makeatother

\newcommand{\Grad} {\ensuremath{\nabla}}

\newcommand{\nor}[1]{\left\| #1 \right\|}
\newcommand{\snor}[1]{\left| #1 \right|}
\newcommand{\LRp}[1]{\left( #1 \right)}

\newcommand{\LRa}[1]{\left< #1 \right>}
\newcommand{\LRc}[1]{\left\{ #1 \right\}}

\renewcommand{\H}{H}
\newcommand{\Hone}{\H^1}

\newcommand{\C}{C}
\renewcommand{\L}{L}

\newcommand{\mc}[1]{\mathcal{#1}}

\renewcommand{\a}{\mathsf{a}}
\newcommand{\h}{h}

\renewcommand{\v}{v}
\newcommand{\vh}{\v^h}
\newcommand{\w}{w}
\newcommand{\wh}{\w^h}
\newcommand{\W}{W}
\newcommand{\Wh}{\W^h}

\newcommand{\DtN}{\Lambda}
\newcommand{\Pih}{\Pi^\h}

\newcommand{\Kn}{\mathsf{Kn}}

\newcommand{\rd}{\mathrm{d}}
\newcommand{\DtNh}{\Lambda^h}
\newcommand{\DtNa}[1]{\DtN_{#1}}

\newcommand{\DtNha}[1]{\DtNh_{#1}}
\newcommand{\DtNR}[1]{\tilde{\DtN}_{#1}}
\newcommand{\DtNRhat}[1]{\hat{\DtN}_{#1}}
\newcommand{\DtNhR}{\hat{\DtN}}
\newcommand{\DtNhRa}[1]{\tilde{\DtN}^h_{#1}}

\newcommand{\Nel} {\ensuremath{{N_\text{el}}}}
\newcommand{\K}{K}
\newcommand{\pK}{{\partial\K}}
\newcommand{\Kj}{{\K_j}}
\newcommand{\pOmega}{{\partial \Omega}}

\newcommand{\V}{{V}}

\newcommand{\Vh}{\V^h}
\newcommand{\VhO}{\Vh_0}
\newcommand{\phih}{\phi^\h}
\newcommand{\eval}[2][\right]{\relax
\ifx#1\right\relax \left.\fi#2#1\rvert}
\newcommand{\Poly}{{\mc{P}}}
\renewcommand{\P}{{\mathbb{P}^h}}
\newcommand{\I}{{\mathbb{I}}}

\renewcommand{\L}{L}

\newcommand{\Phih}{\Phi^\h}
\newcommand{\Phit}{\tilde{\Phi}}

\newcommand{\eqnlab}[1]{\label{eq:#1}}
\newcommand{\theolab}[1]{\label{theo:#1}}

\newcommand{\lemlab}[1]{\label{lem:#1}}

\newcommand{\theoref}[1]{\ref{theo:#1}}

\newcommand{\lemref}[1]{\ref{lem:#1}}
\newcommand{\eqnref}[1]{\eqref{eq:#1}}
\newcommand{\seclab}[1]{\label{sect:#1}}
\newcommand{\secref}[1]{\ref{sect:#1}}

\usepackage{enumitem,comment}

\begin{document}

\maketitle

\begin{abstract}
In computational PDE-based inverse problems,
a finite amount of data is collected to infer unknown parameters in the PDE. In order to obtain accurate inferences, the collected data must be informative about the unknown parameters. How to decide which data is most informative and how to efficiently sample it, is the notoriously challenging task of optimal experimental design (OED). In this context, the best, and often infeasible, scenario is when the full input-to-output (ItO) map, i.e., an infinite amount of data, is available:~This is the typical setting in many theoretical inverse problems, which is used to guarantee the unique parameter reconstruction. 
These two different settings have created a gap between computational and theoretical inverse problems, where finite and infinite amounts of data are used respectively. In this manuscript we aim to bridge this gap while circumventing the OED task. This is achieved by exploiting the structures of the ItO data from the underlying inverse problem, using the electrical impedance tomography (EIT) problem as an example. To accomplish our goal, we leverage the rank-structure of the EIT model, and formulate the ItO matrix\textemdash the discretized ItO map\textemdash as an $\mathcal{H}$-matrix whose off-diagonal blocks are low-rank. This suggests that, when equipped with the matrix completion technique, one can recover the full ItO matrix, with high probability, from a subset of its entries sampled following the rank structure: The data in the diagonal blocks is informative and should be fully sampled, while data in the off-diagonal blocks can be sub-sampled. This recovered ItO matrix is then utilized to present the full ItO map up to a discretization error, paving the way to connect with the problem in the theoretical setting where the unique reconstruction of parameters is guaranteed. This strategy achieves two goals:
\texttt{I)}{\em it bridges the gap between the finite- and infinite-dimensional settings for numerical and theoretical inverse problems} and \texttt{II)} {\em it improves the quality of computational inverse solutions}. We detail the theory for the EIT model, and provide numerical verification to both EIT and optical tomography problems.

\end{abstract}



\section{Introduction}
\seclab{intro}
Inverse problems\textemdash inferring unknown parameters in physical systems from indirect observations\textemdash are ubiquitous in engineering and all branches of sciences. The development of a deep theoretical understanding \cite{Uhlmann_annals,NakamuraUhlmann93,SylvesterUhlmann90} coupled with  the development of highly sophisticated algorithmic pipelines \cite{Tarantola84,MosegaardTarantola95,SymesCarazzone91} for solving inverse problems have fueled several breakthroughs in a myriad of different fields such as geophysics, astronomy, biomedical imaging, radar, spectrography, signal processing, communications, among many others \cite{Cheney_SAR:2001,Virieux_FWI:2017,Biondi:3D_seismic_imaging}. Several of such advances, e.g., magnetic resonance imaging (MRI) \cite{MRI_Schenk:1996}, computarized tomography (CT) \cite{Natterer_CT:2001}, and synthetic aperture radar (SAR) \cite{Cheney_SAR:2001}, permeates the modern life, thus making the study of inverse problems a subject of paramount importance, at both theoretical and algorithmic levels. 

In both theoretical and algorithmic formulations, the object that encodes the accessible knowledge of the unknown parameter is the input-to-output (ItO) map. Although the particular description of this map differs vastly depending on the modeling of the underlying physics, the ItO map generally encodes the impulse response (output) of the medium, or the parameters we seek to reconstruct, from a probing signal (input). At the theoretical level, one assumes that this map is an operator, which maps a functional space of adequate probing signals, to another functional space of the corresponding responses. At the practical and numerical level, there is only a finite
number of possible probing signals that one can use, and the impulse response can only be sampled by a limited amount of receivers, resulting in a finite number of output data. 


Although both theoretical and algorithmic studies seek to shed light on the mechanisms to infer  unknown parameters, they are often not consistent with each other. 
The theoretical study of inverse problems has mainly focused on
answering questions on the infinite-dimensional setting: suppose one knows the full ItO map, can the underlying unknown parameter, living in an infinite-dimensional function space, be uniquely and stably reconstructed? In a nutshell, this infinite-to-infinite approach relies on an infinite amount of data to reconstruct the parameter function, which itself has an infinite number of degrees of freedom.
This infinite-dimensional setting is certainly computationally infeasible. Thus in all algorithmic studies, one focuses on designing algorithmic pipelines to perform the reconstruction on the finite-dimensional setting: given a finite number of ItO measurements, how to extract the information to infer the unknown parameters, represented by finite-dimensional vectors? Unlike the theoretical infinite-to-infinite  approach, this practical finite-to-finite approach uses a finite amount of data pairs to reconstruct the parameter function characterized by a finite number of degrees of freedom.

It is reasonable to believe that the theory should provide guidance and theoretical guarantees for the algorithms' performance. In practice, however, they have mostly advanced in a disconnected manner. Indeed, when one translates the problem from the infinite-dimensional setting to a finite-dimensional one, a large amount of information is often lost. For example, from the theoretical perspective, we do not need to quantify the importance of each data pair since they will be all used. In reality, only a finite amount of data pairs are practically available; thus we ought to select the ones that best inform the parameter reconstruction. Which data pairs are most informative is typically unknown unless an optimal experimental design (OED) (see, e.g., \cite{Pukelsheim06,Chernoff72}) is solved. OED is however notoriously challenging and computationally expensive. 

To bridge the gap between the theoretical study on the infinite-dimensional setting and the numerical study on the finite-dimensional setting, and to maximally use the knowledge from theoretical results, it is necessary to understand the structure of the underlying problem to identify (ideally a small number of) data pairs that are informative about the unknown parameters and then, again exploiting the structure, to lift the information coded in the finite data pairs to the infinite dimensional setting where existing theoretical results are applicable. We stress that finding finitely small information amounts of data is important for real-world applications where data is often expensive and potentially cumbersome to obtain. 

In this manuscript we initiate a line of work to achieve the goal of bridging the gap between theoretical and computational inverse problems. 
In particular, we seek to exploit the structure of the ItO map, and hence the underlying physics of the problem under consideration, to select a subset of \emph{informative} entries in the ItO matrix  to complete the missing entries, thus recovering the full ItO matrix. We then take advantage of the completed ItO matrix in two aspects: 
\begin{enumerate}[label=\Roman*)]
\item {\em bridging the gap between the finite- and infinite-dimensional settings, and}
\item {\em improving the quality of computational inverse solution}.
\end{enumerate}
For {I)}, the completed ItO matrix  is lifted to the ItO map in the infinite dimensional setting where the unique reconstruction of the unknown parameter is guaranteed, which in turn ensures algorithmic convergence.
For {II)}, the completed ItO data matrix is used to reconstruct the parameter through a minimization algorithm. Since completed data contains more information about the parameter than the originally incomplete one, completed data reaches an empirically more accurate inverse solution than its incomplete counterpart. 
Even without matrix completion, the incomplete but informative data facilitate more accurate reconstruction compared to using the same amount of otherwise arbitrary data. 

As have been discussed, the keys to realize our research program are a) the ability to leverage the structure of the underlying problem to sample only a fraction of data and to complete the missing ones, and b) the availability of theoretical results of infinite dimensional inverse problems. The actual executions are thus problem-specific. We choose electric impedance tomography (EIT) for this paper. 
The ItO map in this case is the Dirichlet-to-Neumann (DtN) map and its discretization, the DtN matrix, which possesses 
a $\mathcal{H}$-matrix structure. This allows us to predict the relative importance of entries of the DtN matrix, and sample them accordingly following the $\mathcal{H}$-matrix partitioning. The selected entries are then used to uncover the missing ones through a matrix-completion algorithm \cite{Recht_simple}. Finally, the completed DtN matrix is lifted to the DtN map, allowing us to integrate the theoretical results~\cite{allessandrini_lipschitz,Rondi_e_to_N} on EIT to show the convergence and uniqueness of the reconstructed parameter with a high probability.


\textbf{Outline:} This paper is  organized as follows: we present the whole bridging framework including the data selection process, finite element discretization, a matrix completion algorithm, and a generic parameter reconstruction in Section~\ref{sec:algorithm}. Rigorous results justifying the framework for the EIT problem are presented in Section~\ref{sec:property}. Various numerical results, including showing improved inverse solutions using data completion, are presented in Section~\ref{sec:numerics} to validate our approach for both EIT and optical tomography problems. Section~\ref{sec:conclusions} concludes the paper with future works.

We point out that despite the matrix completion process being almost absent from the inverse problem literature, it was used to solve PDEs in the forward problems~\cite{Solving_PDE_manifold}. Moreover, data-driven approaches have been widely used, in which the most notorious example is the application of compressed sensing to MRI~\cite{Compressed_Sensing_for_MRI}, which already has commercial applications~\cite{CS_MRI_siemmens}.


\section{Bridging framework driven by EIT}\label{sec:algorithm}

Throughout the paper we use the Calder\'on problem as the motivating example. This is considered as a model problem from EIT, in which the voltage is applied on the surface of tissues, and the electric intensity is measured on the surface. By changing voltage configurations, many sets of voltage-to-intensity files can be obtained to infer the conductivity of the medium in the tissue. Mathematically, this translates to utilizing the Dirichlet-to-Neumann (DtN) map to reconstruct the diffusion coefficient in the elliptic equation,
\begin{equation}\label{eqn:elliptic}
\left \{ \begin{array}{rc}
-\nabla\cdot (a(x)\nabla u) = 0,& x\in\mathcal{D}\subset\mathbb{R}^n, \\
u|_{\partial\mathcal{D}} = \phi, & 
\end{array} \right .
\end{equation}
where the input $\phi$ serves as the Dirichlet boundary condition (voltage). The output is also taken on the boundary, and is of Neumann type (electric intensity):
\[
d = a\partial_n u|_{\partial\mathcal{D}}\,,
\]
where $n$ stands for the unit outward normal direction on $\partial\mathcal{D}$. The ItO map from $\phi$ to $d$ is thus termed the voltage-to-intensity map, or mathematically, the DtN map. This map is parameterized by the medium conductivity  $a(x)$:
\[
\Lambda_a: \phi \to d\,,
\]
where the dependence on the conductivity is reflected in the subscript.

Assume $\mathcal{D}$ is a polygonal domain and $\phi \in H^{1/2}(\partial\mathcal{D})$, 
the weak formulation of \eqref{eqn:elliptic} reads: Find $u \in H^1(\mathcal{D})$ with $u|_{\partial\mathcal{D}} = \phi$ such that 
\begin{equation}\label{eqn:ellipticWeak}
\int_{\mathcal{D}}a\,\Grad u \cdot \Grad \v\,\rd x = 0\,, \quad \forall \v \in \Hone_0\LRp{\mathcal{D}}\,.
\end{equation}
The DtN map $\DtNa{a}\phi$ is defined as the following bilinear form
\begin{equation}\label{eqn:DtNa}
\LRa{\DtNa{a}\phi,\psi}:= \int_{\partial\mathcal{D}}a\partial_n u \psi\,\rd x\,,
\end{equation}
where $u$ solves~\eqref{eqn:ellipticWeak}. Using Green's identity and \eqref{eqn:ellipticWeak}:
\begin{equation*}
\LRa{\DtNa{a}\phi,\psi}:= \int_{\mathcal{D}}a\,\Grad u \cdot \Grad \Psi\,\rd x\,.
\end{equation*}
Here $\Psi \in H^1\LRp{\mathcal{D}}$ can be any extension of $\psi$ such that $\eval{\Psi}_{\partial\mathcal{D}} = \psi$. For the rest of the paper, $a(x)$ is assumed to be piece-wise constant 
 and is represented uniquely by the vector $\mathsf{a}$ containing its values, 
and we thus use $a$ and $\mathsf{a}$ interchangeably.

\subsection{Sketch of the DtN map discretization hierarchy} In the numerical setup, the solution and the measurements are all discretized and represented by finite dimensional vectors. The discrete DtN map $\DtNha{\a}$ is therefore a matrix. In this context, the measurements can be viewed as entries in this matrix. 
Only a small number of measurements are taken in experiments, meaning a small number of the entries in the DtN matrix are available. In other words, a subset of entries $\Omega$ of $\DtNha{\a}$ are observed and the rest are unavailable.
This is translated in the following hierarchy of increasingly reduced objects
\begin{equation}
\eqnlab{DtNhierarchy}
\Lambda_{{\a}}\to\Lambda_{\mathsf{a}}^h\to\Lambda^h_{\mathsf{a}}|_\Omega\,.
\end{equation}
Here, again, $\Lambda_a$ is the DtN map, $\Lambda_{\mathsf{a}}^h$ the DtN matrix (discretization of $\DtNa{\a}$, whose size depends on $h$, the mesh size), and $\Omega$ is a subset of matrix indices, indicating where measurements are taken. This reduction process is described in details in section~\ref{sec:setup}.

\subsection{Sketch of reversing the DtN map discretization hierarchy} 
\seclab{reverse}
For the reconstruction, we aim to reverse the hierarchy in \eqnref{DtNhierarchy}. In particular, we start from $\Lambda^h_{\mathsf{a}}|_\Omega$, and by choosing proper data and proper completion algorithms we obtain the full DtN matrix $\DtNha{\a}$. This then gets lifted up to represent the DtN map $\DtNa{\a}$ up to a discretization error that depends on $h$. Due to the involvement of discretization and reconstruction error, the exact recovery of $\DtNha{\a}$ (and hence $\DtNa{\a}$) is not available. We denote $\tilde{\Lambda}_{\mathsf{a}}^h$, $\DtNR{\a}$ and $\tilde{\mathsf{a}}$ the reconstructed approximations to $\DtNha{\a}$, $\DtNa{\a}$, and $\a$ respectively. $\Omega$ is judiciously selected, such that

\begin{equation}\label{eqn:close_recons_matrix}
\tilde{\Lambda}_{\mathsf{a}}^h\sim \Lambda_{\mathsf{a}}^h \,,
\end{equation}
where $\sim$ means close in some sense (to be defined later).
For small $h$, we lift the matrix back to the map and need to justify:
\begin{equation}\label{eqn:close_recons_h}
\DtNR{\a} \sim \Lambda_{{a}}. 
\end{equation}
Finally 
we seek to establish the closeness of the reconstruction of the media:
\begin{equation}\label{eqn:close_recons_a}
\tilde{\Lambda}_{\mathsf{a}}^h\sim \Lambda_{\mathsf{a}}^h\, \quad \Rightarrow \quad \DtNR{\a} \sim \Lambda_{{a}}\,\quad\Rightarrow\quad \tilde{\mathsf{a}}\sim {a}\,.
\end{equation}

In Section~\ref{sec:setup} we lay out the numerical setup. We recover $\tilde{\Lambda}^h_{\mathsf{a}}$ from the subsampled $\Lambda^h_{\mathsf{a}}|_\Omega$, and provide intuition to~\eqref{eqn:close_recons_matrix} in Section~\ref{sec:recons_map}. Recall from section \secref{intro} that the construction of $\tilde{\Lambda}^h_{\mathsf{a}}$ is twofold: \texttt{I)}{\em bridging the gap} and \texttt{II)} {\em improving the quality of computational inverse solution}. 
The proofs for ~\eqref{eqn:close_recons_matrix}, ~\eqref{eqn:close_recons_h}, and~\eqref{eqn:close_recons_a} are given in Section~\ref{sec:property}, which accomplish task  \texttt{I)} of {\em bringing the gap between theoretical and computational EIT.}
Section~\ref{sec:recons_parameter} discusses a practical computational algorithm for task \texttt{II)} which aims to approximately reconstruct $\a$ from the DtN matrix $\tilde{\Lambda}^h_{\mathsf{a}}$.  It is important to point out that\textemdash unlike traditional computational inverse problems that uses  $\Lambda^h_{\mathsf{a}}|_\Omega$, the incomplete DtN matrix, to reconstruct $\a$\textemdash we deploy $\tilde{\Lambda}^h_{\mathsf{a}}$, the completed DtN matrix, to reconstruct $\a$. As shall be shown 
in section \ref{sec:numerics}, our approach {\em  improves the parameter reconstruction substantially}. Indeed, the reconstructions using $\tilde{\Lambda}^h_{\mathsf{a}}$ and the exact DtN matrix $\DtNha{a}$ are visibly identical while the reconstruction directly from $\Lambda^h_{\mathsf{a}}|_\Omega$ is completely off.

\subsection{DtN map discretization hierarchy}\label{sec:setup}
In what follows we provide details of the DtN map discretization hierarchy \eqnref{DtNhierarchy}.

\noindent \textbf{From $\Lambda_a$ to $\Lambda_{\mathsf{a}}^h$}: Numerically, we first partition the domain $\mc{D}$ into $\Nel$ non-overlapping shape-regular affine elements $\Kj, j = 1,\hdots,\Nel$ with Lipschitz boundaries. Denote $\mc{D}^h := \cup_{j=1}^\Nel \Kj$, $\overline{\mc{D}} = \overline{\mc{D}}^h$ the discrete space, and $h=\max_{j}\text{diam}\LRp{\Kj}$ the mesh size, we construct the standard linear Lagrange finite element (FE) space
\[
 \Vh := \LRc{v \in C^0(\mathcal{D}): \eval{v}_{K_j} \in \Poly^1(K_j)\,,\forall j} \subset \C^0\LRp{\mc{D}} \subset \Hone\LRp{\mc{D}} \,,
\]
as the discrete solution space, and
\[
\Vh\LRp{\partial\mc{D}} := \LRc{v \in C^0\LRp{\partial\mc{D}}: \eval{v}_{\pK\cap\partial\mc{D}} \in \Poly^1\LRp{\pK\cap\partial\mc{D}}} = \text{span}\{\phi_i\}
\]
as the discrete boundary condition space where $\phi_i$ are the linear nodal (Lagrange) basis functions on $\partial\mc{D}$. Here $\Poly^1\LRp{\K}$ is the space of polynomials of degree at most $1$ on $\K$. To project the boundary condition from the continuous level to the discrete one, we define the projection operator
\begin{equation}\label{eqn:projection}
\Pih: \H^{1/2}\LRp{\partial\mc{D}} \ni \phi \mapsto 
\phih:= \Pih\phi\in \Vh\LRp{\partial\mc{D}}
\end{equation}
such that
\[
\nor{\phi - \phih}_{\H^{1/2}\LRp{\partial\mc{D}}} = \nor{\phi - \Pih\phi}_{\H^{1/2}\LRp{\partial\mc{D}}} = \inf_{\wh \in \Vh\LRp{\partial\mc{D}}} \nor{\phi - \wh}_{\H^{1/2}\LRp{\partial\mc{D}}}\,.
\]
 The discretization of the weak formulation~\eqref{eqn:ellipticWeak} reads: Find $u_h \in V^h$ such that $\eval{u^h}_{\partial\mc{D}} = \phih=\Pih \phi$, and   
\begin{equation}\label{eqn:FEMeqn}
\int_{\mc{D}}\a\,\Grad u^h \cdot \Grad \vh\,\rd x = 0, \quad \forall \vh \in \VhO\,,
\end{equation}
where $\VhO := \LRc{\v \in \Vh\LRp{\mc{D}}: \eval{\v}_{\partial\mc{D}} = 0}$.

The discretized DtN map $\DtNha{a}$ is bilinear on $\Vh\LRp{\pOmega}$: for $\phih \in  \Vh\LRp{\pOmega}$ and $\wh \in \Vh\LRp{\pOmega}$
\begin{equation}
\eqnlab{DtNha}
\LRa{\DtNha{a}\phih,\wh}:=
\int_{\Omega}a\,\Grad \Phi^h \cdot \Grad \Wh\,d\Omega,
\end{equation}
where $\Wh$ is any extension of $\wh$ from $\Vh\LRp{\pOmega}$ to $\Vh\LRp{\Omega}$ such that $\eval{\Wh}_{\pOmega} = \wh$ and $\Phi^h$ is the FE solution obtained from \eqref{eqn:FEMeqn}. The DtN matrix is the matrix representation of $\DtNha{a}$, and we abuse the notation and still call it $\DtNha{a}$. It can be constructed as follows.
Let $\mathsf{d}$ be the FEM discretization of $d$, and $\mathsf{S}$ the FEM stiffness matrix, then the numerical solution $\mathsf{u}^h$ is 
\[
\mathsf{u}^h = \left(\mathsf{S}^{ii}\right)^{-1}\cdot\mathsf{S}^{ib}\cdot\phi^h\,,
\]
where 
\[
\mathsf{S} = \left [ \begin{array}{cc} \mathsf{S}^{ii} & \mathsf{S}^{ib}\\
                                   \mathsf{S}^{bi} & \mathsf{S}^{bb}
                     \end{array}\right], 
\]
$i$ stands for the collection of the indices of the interior degrees of freedom, and $b$ is for the degrees of freedom at the boundary. Furthermore, denote $\mathsf{M}$ the map from the discrete solution $\mathsf{u}^h$ to the discrete Neumann data on the boundary $\partial\mathcal{D}^h$,  the DtN matrix $\Lambda^h_{\mathsf{a}}$ can be formed as
\begin{equation}\label{eqn:lambda_h_a_formula}
\mathsf{d} = \Lambda^h_{\mathsf{a}} \cdot\phi = \mathsf{M}\cdot\left(\mathsf{S}^{ii}\right)^{-1}\mathsf{S}^{ib}\cdot\phi\,,\quad\text{with}\quad \Lambda^h_{\mathsf{a}} = \mathsf{M}\cdot\left(\mathsf{S}^{ii}\right)^{-1}\mathsf{S}^{ib}\,.
\end{equation}
Note that $\Lambda^h_{\mathsf{a}}$ is a square matrix of size $|\partial\mathcal{D}^h|\times|\partial\mathcal{D}^h|$, where 
$\snor{\partial\mathcal{D}^h}$ denotes the number of  grid points on $\partial\mathcal{D}^h$.


\noindent \textbf{From $\Lambda_{\mathsf{a}}^h$ to $\Lambda_{\mathsf{a}}^h|_\Omega$}: 
In practice, only a small number experiments can be conducted, and in each experiment, only  a small number of measurements can be  taken. For notational convenience, we assume the input $\phi$ is chosen from the set basis functions $\phi_i$. 
In this case the 
$ij$th component of the data matrix $\mathsf{d}$ is exactly the $ij$th entry of the DtN matrix, i.e.,
\[	
d_{ij} \sim  \Lambda^h_{\mathsf{a},ij}\,,\quad (i,j)\in \Omega\subset [1:|\partial\mc{D}^h|]^2\,.
\]
Here we use $\sim$ instead of $=$ to account for potential measuring errors, and $\Omega$, referred to as a mask, is a subset of all indices of $\Lambda^h_{\mathsf{a}}$.


\subsection{Reverse DtN map discretization hierarchy}\label{sec:recons_map}

As it was argued in section \secref{intro}, some data pairs are more informative than the others. Choosing the most informative data, or equivalently, selecting the right mask $\Omega$, is of paramount importance in recovering the missing entries in the DtN matrix $\DtNha{a}$. Recall from section \secref{reverse} that, due to errors in the discretization and reconstruction process, we can only obtain an approximation $\tilde{\Lambda}^h_{\mathsf{a}}$ of $\DtNha{a}$. In the following we exploit the structure of the DtN matrix $\DtNha{a}$ to determine $\Omega$ and employ a matrix completion technique such that
 $\tilde{\Lambda}^h_{\mathsf{a}}$ is close to ${\Lambda}^h_{\mathsf{a}}$.

To describe the structure of the DtN matrix, we exploit the concept of $\mathcal{H}$-matrices. With a proper decomposition, the DtN matrix can be partitioned into several low-rank blocks regardless of their size. This decomposition allows us to utilize the matrix completion type methods for the low rank blocks that are not applicable to the full DtN matrix as it is often of full rank. In the following we briefly review the matrix completion method in section~\ref{sec:matrix_completion}, and evaluate the matrix structure of $\Lambda_{\mathsf{a}}^h$ in section~\ref{sec:H_matrix}. The full completion algorithm is presented in Algorithm~\ref{alg:reconstructh}.


\subsubsection{Matrix completion}\label{sec:matrix_completion}

Matrix completion has been a popular topic for a decade due to its applications in recommendation systems, including the famous Netflix problem \cite{NetflixProblem}. The goal is to complete the entries in a matrix from a partial knowledge of its entries. To be more specific, let a generic $\mathsf{A} \in \mathbb{R}^{n \times n}$ be the to-be-completed matrix, of which only some of its entries are known. In this setting $\Omega$, with $|\Omega|=m$, is the index set where the entries are known, and $a_{ij}$ the given values with $(i,j)\in\Omega$.

There exist a number of algorithms that aims to reconstruct the entries \cite{Survey_matrix_completion,Hofmann2004LatentSM,Goldberg92usingcollaborative,Bui-ThanhDamodaranWillcox04}. We adopt the approach proposed in~\cite{matrix_completion_original}. Under the assumption that the matrix $\mathsf{A}$ is of low rank ($r\ll n$), we seek to minimize the nuclear norm $\|\mathsf{A}\|_\ast$, such that the matrix's evaluations at certain locations are fixed by the given data, namely $\mathsf{A}_{ij}=a_{ij}$ for $(i,j)\in\Omega$. The minimization problem now becomes:
\begin{equation}\label{eqn:opt_nuclear}
\min_\mathsf{A} \|\mathsf{A}\|_\ast\,,\quad\text{s.t.} \quad \mathsf{A}_{ij} = a_{ij}\,, \quad (i,j)\in\Omega\,,\quad |\Omega|=m\,.
\end{equation}
The objective function $\|\mathsf{A}\|_\ast$ is the sum of all singular values of $\mathsf{A}$. It can be viewed as the relaxation from $\ell_0$-norm of singular values ($\text{rank}\{\mathsf{A}\}$) to its $\ell_1$-norm~\cite{L1L0_matrix}.

One important advantage of working with~\eqref{eqn:opt_nuclear}
is that it is a convex optimization problem that can be solved efficiently with interior point methods. It is important to point out that this convex relaxation does not come with an accuracy sacrifice: it is found that under very mild conditions\textemdash\emph{decoherent} and \emph{delocalization} conditions\textemdash the solution of minimizing $\|\mathsf{A}\|_\ast$ coincides with the solution of minimizing $\text{rank}\{\mathsf{A}\}$. 

We next recall the \emph{decoherent} and \emph{delocalization} conditions~\cite{matrix_completion_original}. 
Let
\[
\mathsf{A} = \mathsf{U}\Sigma \mathsf{V}^\top\,
\]
be the singular value decomposition of $\mathsf{A}$.
\begin{definition}
Let $\mathsf{W}$ be a subspace of $\mathbb{R}^n$ of dimension $r$ and $\mathcal{P}_{\mathsf{W}}$ be the orthogonal projection onto $\mathsf{W}$. Then the coherence index of $\mathsf{W}$ is defined as
\begin{equation}\label{eqn:def_coherence}
\mu(\mathsf{W}) = n \max_{1\le i \le n} \|\mathcal{P}_{\mathsf{W}}e_i\|_2\,,
\end{equation}
where $e_i$ is $i$-th unit vector of $\mathbb{R}^n$.
\end{definition}
\begin{enumerate}[label = \textbf{A\arabic*}]
\item \label{enum:decoherence} \underline{Decoherent condition}: $\max(\mu(U),\mu(V))\leq \mu_0$ for some positive $\mu_0$. 
\item \label{enum:sparsity} \underline{Delocalization condition}: 
The maximum entry of $\sum_{1\leq k\leq r} u_kv_k^\top$ is bounded from above by $\mu_1\sqrt{\frac{r}{n^2}}$ for some positive $\mu_1$.
\end{enumerate}
Let us now state a probabilistic result on the success of~\eqref{eqn:opt_nuclear}:
\begin{theorem}[\cite{matrix_completion_original,Recht_simple}]\label{thm:mc}
Let $\mathsf{A}$ be an $n\times n$ matrix of rank $r$ obeying the decoherent conditions \ref{enum:decoherence} and \ref{enum:sparsity}. Suppose we observe $m$ entries of $\mathsf{A}$ with locations sampled uniformly at random. Then there exist constants $C$, $c$ such that if
\[
m \geq C \max\LRc{\mu^2_1, \mu^{1/2}_0\mu_1 ,\mu_0n^{1/4}} \beta r \left (n \log n \right)
\]
for some $\beta > 2$, then the minimizer to the problem~\eqref{eqn:opt_nuclear} is unique and equal to $\mathsf{A}$ with probability at least $1- cn^{-\beta}$. For $r \leq \mu_0^{-1}n^{1/5}$ this estimate can be improved to
\[
m \geq C \mu_0  \beta r  \left ( n^{6/5} \log n\right ) \,,
\]
with the same probability of success.
\end{theorem}

This theorem suggests that if the to-be-completed matrix $\mathsf{A}$ is of low rank satisfying the decoherent and delocalization condition, then with the number of provided entries $m$ linearly depend on $r$, the rank, and the entries are sampled uniformly randomly, the matrix can be precisely reconstructed with a high probability.

\subsubsection{Structure of $\Lambda_{\mathsf{a}}^h$ and $\mathcal{H}$-matrix}\label{sec:H_matrix}
We now discuss how we use the matrix completion method discussed in \ref{sec:matrix_completion} to recover the DtN matrix $\Lambda^h_{\mathsf{a}}$ from its data $\Lambda^h_{\mathsf{a}}|_\Omega$.
Theorem \ref{thm:mc}, despite providing a general recipe for reconstructing a matrix from its incomplete data, requires the rank $r$ to be significantly smaller than $n$, the size of the matrix, for the algorithm to be meaningful. However, $\Lambda^h_\mathsf{a}$ is a full-rank matrix, preventing the direct application of the matrix completion algorithm.

It turns out that we can still take advantage of the matrix completion algorithm by exploiting the $\mathcal{H}$-matrix structure embedded in $\Lambda_{\mathsf{a}}^h$.  This allows us, through the ``peeling" process~\cite{LinLuYing}, to divide the matrix into sub-blocks, most of which are low-rank. The application of the matrix completion algorithm to these low-rank blocks is then expected to be efficient.

Hierarchical matrix, or commonly referred as $\mathcal{H}$-matrix, is a class of matrices that, upon proper partitioning, have fast decays in singular values in the smaller blocks, leading to the low rank property within these blocks. The concept ~\cite{Hackbusch_1,Hackbusch_2} was invented initially to divide a given $\mathcal{H}$-matrix into smaller blocks, so that some matrix operations, including matrix-vector multiplication, addition, inverse,  Schur complement and many others, could be significantly sped up 
(see, e.g., ~\cite{bebendorf2008hierarchical} and references therein). In its original form, one also requires linear computational complexity in finding the partitioning. In our setting, such complexity is irrelevant. We point out that several matrix compression methods can be used to reconstruct low-rank blocks, such as adaptive cross approximation \cite{Bebendorf11},  randomized SVD \cite{Tropp_rSVD}, 
or CUR factorizations \cite{Mahoney:CUR,Goreinov:pseudoskeleton}.  However, they require to sample full rows and columns, or they require matrix-vector multiplications,
and thus are not applicable for our context in which we have access to only a subset of entries. 

It was shown in~\cite{bebendorf2008hierarchical} that the collection of Green's functions for elliptic equations produces an $\mathcal{H}$-matrix. Recall from the expression for $\Lambda^h_{\mathsf{a}}$ in~\eqref{eqn:lambda_h_a_formula}
that both $\mathsf{M}$ and $\mathsf{S}^{ib}$ are sparse matrices. Thus if $\left(\mathsf{S}^{ii}\right)^{-1}$ is an $\mathcal{H}$-matrix, so is $\Lambda^h_{\mathsf{a}}$. 

\begin{theorem}[Theorem 4.28 of~\cite{bebendorf2008hierarchical}]
Let $\epsilon>0$ small, then there is an $\mathcal{H}$-matrix $\mathsf{C}$ with local block rank being $k\lesssim \left(\log n\right)^2|\log\epsilon|^{d+1}$, such that $\|\left(\mathsf{S}^{ii}\right)^{-1} - \mathsf{C}\|_2 \leq \epsilon$. For the accuracy compatible with the finite element method, we take $\epsilon = h^\beta$, and $k\lesssim \log^{d+3}n$. Here $n$ is the size of the matrix, $d$ is the dimension of the problem, $h$ is the mesh size and $\beta$ is the accuracy order of the finite element method.
\end{theorem}

The author of \cite{bebendorf2008hierarchical}  furthermore suggests a way to choose the partition. Indeed $\left(\mathsf{S}^{ii}\right)^{-1}$ is essentially the discrete version of the Green's function $G(x,y)$, and as a function of $x$ parameterized by $y$ it
can be approximately written as a summation of separable functions in $x$ and $y$ if the two coordinates are well separated. 
This reveals that the rank of the approximation only logarithmically depends on the size of the matrix, and since the Green's function can be written as separable function only if $x$ and $y$ are well-separated, the rank is low only for the blocks of $\left(\mathsf{S}^{ii}\right)^{-1}$ that are not along the diagonal. The same observation was made in~\cite{LinLuYing} which shows that the off-diagonal blocks are of low rank.


Built upon these observation, noting that $\Lambda^h_\mathsf{a}$ is an $\mathcal{H}$-matrix, with diagonal blocks, $\Lambda^h_{\mathsf{a},D}$, having the full rank, and the off-diagonal blocks, $\Lambda^h_{\mathsf{a},O}$ being approximately low-rank, 
we propose to obtain the full data in $\Lambda^h_{\mathsf{a},D}$ but a limited entries in $\Lambda^h_{\mathsf{a},O}$, and then recover the missing entries with the matrix completion method in section~\ref{sec:matrix_completion}.
Experimentally, this means for every injected voltage concentrated on one spot on $\partial\mathcal{D}$, one measures the intensity on that particular spot to fill in the diagonal blocks $\Lambda^h_{\mathsf{a},D}$. One then decreases the density of the detectors as one moves further away along $\partial\mathcal{D}$, and employs the matrix completion algorithm to recover $\Lambda^h_{\mathsf{a},O}$. In Algorithm~\ref{alg:reconstructh}, we summarize the whole process of completing $\Lambda^h_{\mathsf{a}}$.

\begin{algorithm}[h!t!b!]
\caption{\textbf{Completing $\Lambda^h_{\mathsf{a}}$}}\label{alg:reconstructh}
\begin{algorithmic}
\State \textbf{Preparation:}
\State 0. Determine the partition, and identify $\Lambda^h_{\mathsf{a},D}$ and $\Lambda^h_{\mathsf{a},O}$. 
\State 1. Sample each entry in $\Lambda^h_{\mathsf{a},D}$.
\State 2. For each $\Lambda^h_{\mathsf{a},O}$:
\State 2.1: Randomly collect $r n^{6/5}\log{n}$ data points in the block ($n$: the size of the block);
\State 2.2: Solve the matrix completion problem \eqref{eqn:opt_nuclear} to reconstruct $\Lambda^h_{\mathsf{a},O}$;
\State \textbf{end}
\State \textbf{Output:} Assembled the recovery, denoted by $\tilde{\Lambda}^h_{\mathsf{a}}$.
\end{algorithmic}
\end{algorithm}

Once $\tilde{\Lambda}^h_{\mathsf{a}}$ is formed, it can be lifted to a corresponding DtN map $\DtNR{\a}$ which in turn corresponds to a unique conductivity $\tilde{\a}$. The analysis on the difference between $\DtNa{a}$ and $\DtNR{\a}$ (and between $\a$ and $\tilde{\a}$) is presented in \ref{sec:property}.

\subsection{Improving inverse solution with matrix completion}\label{sec:recons_parameter}

With the full map $\DtNhRa{\a}$ in hand, the reconstruction of the media $\mathsf{a}$ is now straightforward using classical optimization-based methods. Since this component of the algorithm is rather classical, we briefly review it here. 

We consider the reconstructed $\tilde{\Lambda}_{\mathsf{a}}^h$ as the groundtruth data, and we search for the media impedance $\mathsf{a}$ such that the misfit\textemdash with Frobenius norm\textemdash  between the DtN matrix generated by that $\mathsf{a}$ and the groundtruth data is minimized, i.e.,
\begin{equation}\eqnlab{eqn:opt_a}
\min_{\mathsf{a}} \|\Lambda^h_{\mathsf{a}}-\tilde{\Lambda}^h_\mathsf{a}\|^2_F\, +  \alpha\|\mathsf{a}-\mathsf{a}_0\|^q_q + \beta R(\mathsf{a})\,,
\end{equation}
where the second term is a regularization term taking into account some prior knowledge,  and the third term is an additional regularization term  to enforce desirable properties in the reconstruction.
Note that even though the first term may seem benign at first glance, the DtN map, $\Lambda^h_{\mathsf{a}}$, is highly non-linear in $\a$. This may imply the existence of many local minima in the objective function landscape, which can greatly tax the capability of standard gradient-based optimization techniques. In this context, both regularization terms can be tuned to attenuate this issue, however, how to tune these methods is outside the scope of this paper. For our numerical results in section \ref{sec:numerics}, we set both regularization parameters $\alpha, \beta$ to zero, and we will use an off-the-shelf optimization 
Gauss-Newton method to reconstruct  the impedance $\a$. 



\section{Bridging the gap with matrix completion}\label{sec:property}

To show~\eqref{eqn:close_recons_matrix} amounts to showing the reconstructed DtN matrix $\tilde{\Lambda}^h_{\mathsf{a}}$ is close to the true  DtN matrix $\DtNha{\a}$. For that we combine the $\mathcal{H}$-matrix argument and the matrix completion result.

\begin{theorem} \label{thm:reconstruction}
Divide $\DtNha{\a}$ into $N$ blocks with each of size $n\times n$ according to the $\mathcal{H}$-matrix decomposition. Suppose the $i$-th block has rank $r^i$ and obeys the decoherent and delocalization conditions with constants $\mu_0^i$ and $\mu_1^i$. Denote $m^i$ the number of observed entries in the $i$-th block with samples chosen uniformly at random. Then to reconstruct $\Lambda^h_{\mathsf{a}}$ using~\eqref{eqn:opt_nuclear}, there exist constants $C$, $c$ such that if
\[
m^i \geq C \max\LRc{\LRp{\mu^i_1}^2, \sqrt{\mu^i_0}\mu_1^i,\mu_0^in^{1/4}} \beta^i r^i \left (n \log n \right), \quad i = 1,\hdots,N,
\]
for some $\beta^i > 2$, then
\[
\mathbb{P}(\DtNha{\a}=\tilde{\Lambda}^h_{\mathsf{a}}) \ge 1 - c\sum_{i=1}^Nn^{-{\beta^i}}\,.
\]
The sampling sizes can be improved to $m^i \geq C \mu_0^i  \beta r^i  \left ( n^{6/5} \log n\right )$ if $r^i \leq n^{1/5}/\mu_0^i$.
\end{theorem}
\begin{proof}
The proof is a straightforward application of Theorem \ref{thm:mc}.
\end{proof}

\begin{remark}
We assume that the rank of $i$-th block is $r^i$. One should note that this is only an approximate rank. According to~\cite{BL_2011}, elliptic boundary-to-boundary operators have exponentially decaying singular values and thus $r^i$ depend on the error tolerance. For a more precise reconstruction, larger $r^i$ may be needed, and it amounts to a higher value of $m^i$, meaning more data points are needed.
\end{remark}

\begin{remark} The theorem states that the two matrices, the reconstructed and the ground-truth, are exactly the same with high probability. In practice, the data obtained in $\DtNha{\a}|_\Omega$ is often polluted with measurement errors.
In~\cite{Candes_noise} the authors discuss the effect of such pollution in the reconstruction.
\end{remark}

To quantify ~\eqref{eqn:close_recons_h} and \eqref{eqn:close_recons_a}, we will rely on some delicate FEM analysis and Theorem \ref{thm:reconstruction}. 
To begin, we lift both $\DtNha{\a}$ and $\DtNhRa{\a}$ matrices to their ``corresponding" (or reconstructed) DtN maps $\DtNRhat{\a}$ and $\DtNR{\a}$ as follows:
\begin{equation}
\eqnlab{DtNreconstructed}
\DtNR{\a} := \LRp{\Pih}^*\circ\DtNhRa{\a} \circ\Pih, \quad \DtNRhat{\a}  := \LRp{\Pih}^*\circ\DtNha{\a} \circ\Pih,
\end{equation}
where $\LRp{\Pih}^*$ is the adjoint of $\Pih$. It is easy to see that both $\DtNR{\a}$ and $\DtNRhat{\a}$ are well-defined linear continuous operators from $H^{1/2}\LRp{\pOmega}$ to $H^{-1/2}\LRp{\pOmega}$.
Theorem \ref{thm:reconstruction} implies that $\DtNR{\a}=\DtNRhat{\a}$ with high probability, i.e.,
\[
\mathbb{P}\LRp{\DtNR{\a}=\DtNRhat{\a}} \ge 1 - c\sum_{i=1}^Nn^{-{\beta^i}}.
\]
Thus, the following results, except Theorem \theoref{lip} and Theorem \theoref{complexity}, while being valid deterministically for $\DtNRhat{\a}$, are valid with high probability for $\DtNR{\a}$. 
Let us recall the following well-known result.
\begin{theorem}[\cite{allessandrini_lipschitz,Rondi_e_to_N}]\theolab{lip}
Let $a_1$ and $a_2$ be two piecewise constant functions on $\mc{D}$, and $\mathsf{a}_i$ are their representation vectors. Denote $\Lambda_{\mathsf{a}_i}$ the corresponding DtN map defined in~\eqref{eqn:DtNa}. Then $\a$ as a function of $\DtNa{\a}$ is Lipschitz: 
\begin{equation}\label{eqn:well_cont}
\|\mathsf{a}_1-\mathsf{a}_2\|_\infty \leq C\|\Lambda_{\mathsf{a}_1}- \Lambda_{\mathsf{a}_2}\|_{H^{1/2}\LRp{\pOmega}\to H^{-1/2}\LRp{\pOmega}}\,,
\end{equation}
that is, there is a unique  conductivity $\a$ for every DtN map $\DtNa{\a}$.
\end{theorem}

Since we assume that $\a$ is piecewise constant 
Theorem \theoref{lip} implies that there exists a unique conductivity $\tilde{\a}$ corresponding to $\DtNRhat{\a}$ such that
\[
\nor{\a -\tilde{\a}}_\infty \le C\nor{\DtNa{\a} - \DtNRhat{\a}}_{H^{1/2}\LRp{\pOmega}\to H^{-1/2}\LRp{\pOmega}}.
\]

The following (
whose technical proof is given in the Appendix~\ref{sec:appendix_proof}) is the justification for \eqref{eqn:close_recons_h} and  \eqref{eqn:close_recons_a}, which, similar to Theorem \theoref{lip}, shows the uniqueness the reverse DtN hierarchy in section \ref{sec:recons_map}. 
\begin{theorem}[Asymptotic Uniqueness]
\theolab{asymptoticUniquenessFull}
There holds
\[
\lim_{h\to 0} \nor{\DtNa{\a} - \DtNRhat{\a}}_{H^{1/2}\LRp{\pOmega}\to H^{-1/2}\LRp{\pOmega}} = 0,
\]
and thus
\[
\lim_{h\to 0} \nor{\a -\tilde{\a}}_\infty = 0.
\]
Let $\tilde{\a}_1$ and $\tilde{\a}_2$ be the conductivities associated with two reconstructed DtN maps $\DtNRhat{\a_1}$ and $\DtNRhat{\a_2}$ corresponding to ${\a}_1$ and ${\a}_2$, respectively. Then
\begin{align*}
\nor{\tilde{\a}_1 - \tilde{\a}_2}_\infty \le C\left(\nor{\DtNa{\a_1} - \DtNRhat{\a_1}}_{H^{1/2}\LRp{\pOmega}\to H^{-1/2}\LRp{\pOmega}} \right.\\
\left. +  \nor{\DtNa{\a_1} - \DtNa{\a_2}}_{H^{1/2}\LRp{\pOmega}\to H^{-1/2}\LRp{\pOmega}} + \nor{\DtNa{\a_2} - \DtNRhat{\a_2}}_{H^{1/2}\LRp{\pOmega}\to H^{-1/2}\LRp{\pOmega}}\right).    
\end{align*}
That is, if $\a_1 \to \a_2$ then $\tilde{\a}_1 \to \tilde{\a}_2$ as $h \to 0$.
\end{theorem}

We also briefly discuss the complexity of the matrix completion.
\begin{theorem}\theolab{complexity}
Denote $\DtNa{\a}$ a matrix of size $n\times n$ with $n\sim 1/h^d$, and suppose that the matrix can be decomposed in $\mathcal{H}$-matrix with the weak admissibility condition (see Figure \ref{fig:h_mat_decomposition}). Suppose that the decomposition has $L \sim \log{n}$ levels, and that each block has a bounded rank $r$ and satisfies the conditions of Theorem.~\ref{thm:mc}. Then with $|\Omega|=\mathcal{O} \left (rn^{6/5} \log n \right)$ known entries sampled properly, we can reconstruct $\DtNa{\a}$ with high probability.
\end{theorem}

\begin{proof}
We consider diagonal and off-diagonal blocks separately. For the partition considered in this theorem, there are $n/r$ diagonal blocks. They are full rank, which requires them to be fully sampled, and thus $\mathcal{O}(rn)$ entries are needed.

For the off-diagonal blocks we need to use randomized sampling. Given that the matrix is partitioned in $L$ levels, we have that at the $\ell$-th level in the decomposition, each block will have size $n_{\ell} = n/2^{\ell}$, and we will have $n/n_{\ell} = 2^{\ell}$ of them. 
Thus following Theorem.~\ref{thm:mc}, we require $\mathcal{O} \left( r(n^{\ell})^{6/5} \log \left ((n^{\ell})^{6/5}\right) \right)$, 
or $\mathcal{O}( r(n/2^{\ell})^{6/5} \log (n^{\ell})^{6/5}))$
samples to reconstruct each of the blocks at the $\ell$-th level in the partition. 
In summary, we would require $\mathcal{O}\left( 2^{\ell} r(n/  2^{\ell})^{6/5} \log \left( (n/2^{\ell})^{6/5} \right) \right)$ to reconstruct, with high-probability, all the blocks at the $\ell$-th level of the partition. 

After adding the number of samples at each level we have that the total number of samples required scales as 
\begin{align}
    \mathcal{O} \left ( \sum_{\ell = 1}^{n_{levels}} 2^{\ell} r(n/  2^{\ell})^{6/5} \log{ (n/  2^{\ell})^{6/5}} \right ) & = \mathcal{O} \left ( rn^{6/5} \log{n} \sum_{\ell = 1}^{n_{levels}} 2^{-\ell/6}  \right ), \\ 
     & = \mathcal{O} \left ( rn^{6/5} \log{n} \right),
\end{align}               
where we used the fact that $\sum_{\ell = 1}^{n_{levels}} 2^{-\ell/6} = \mathcal{O}(1)$.
\end{proof}

\begin{figure}[htbp]
\centering

\includegraphics[trim = 29mm 0mm 29mm 10mm, clip, width = \textwidth]{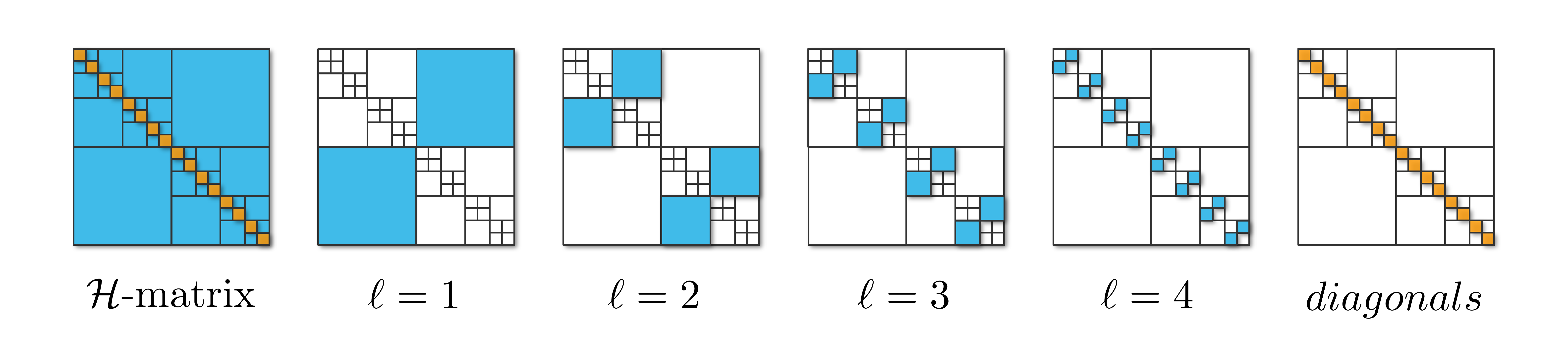}
 \caption{Sketch of a partition of an $\mathcal{H}$-matrix in four levels. Blocks in light-blue  are randomly sampled, and  those in orange are fully samples. We show the blocks considered at each level of the partition. As can be seen, we have $2^{\ell}$ number of blocks of size $n^{\ell} = n/2^{\ell}$.}
\label{fig:h_mat_decomposition}
\end{figure}

\begin{remark}
We have chosen the weak admissibility condition for the sake of simplicity:  the theorem follows for other types of partitioning. In particular, as it will be shown in Section \ref{sec:DtN_numerics}, using a strong admissibility condition with periodic boundary conditions yields similar results. 
\end{remark}

\section{Numerical Experiments}\label{sec:numerics}
We finally present several numerical experiments showcasing the framework introduced above. All the experiments were coded in Matlab 2019b using CVX to solve the optimization problems with Mosek \cite{mosek} as the back-end. The experiments were run on a single-socket workstation running an AMD 2950X processor with 128 GB of RAM. 

As the method suggests, we should first run the hierarchical matrix completion algorithm in Algorithm.~\ref{alg:reconstructh} to reconstruct the full DtN matrix, and then use the completed DtN matrix to reconstruct the media. For an accurate DtN matrix reconstruction, we need to ensure the matrix gets decomposed according to the $\mathcal{H}$-matrix admission condition, and the matrix completion algorithm is implemented within each off-diagonal block that is of low rank. The matrix completion algorithm requires two conditions to be held: the decoherent, and delocalization conditions.

To study the performance of the method, we will first demonstrate that the off-diagonal blocks indeed satisfy the decoherent and delocalization conditions. These conditions will be shown to be satisfied independent of the level of numerical refinement. This ensures that the matrix completion algorithm indeed reconstructs the DtN matrix accurately with limited data. With the demonstration of the accurate reconstruction of the DtN matrix, we further showcase the reconstruction of the media/conductivity. This final result will be compared with the reconstruction obtained by a complete random sampling of the DtN matrix. The comparison suggests samplings that honor the local low-rank structure of the off-diagonal blocks significantly outperforms a blind random sampling strategy.

In subsection~\ref{sec:DtN_numerics}, we verify the decoherent and the delocalization conditions and present  reconstruction of the DtN matrix. In subsection~\ref{sec:media_numerics} we showcase the reconstruction of the conductivity. In subsection~\ref{sec:OT_numerics} we demonstrate an extension of the presented method on optical tomography (OT). In OT, the albedo operator maps the incoming light to the outgoing light intensity, and is used to reconstruct the scattering coefficient, an optical property of the material. 

\subsection{Reconstructing the DtN map}\label{sec:DtN_numerics}
As presented in section~\ref{sec:recons_map} the reconstruction of DtN matrix relies on two key factors: a proper $\mathcal{H}$-matrix decomposition, and the proper use of the matrix completion algorithm in the low-rank sub-matrices that satisfy both decoherent and delocalization conditions.
We demonstrate both the $\mathcal{H}$-matrix decomposition and the final matrix completion results.

We now detail the numerical setup. In  $\mc{D} = [0,1]^2$ domain, we choose 
the Shepp-Logan phantom as the ground-truth media, as plotted in Figure~\ref{fig:impedance_1}. On the domain we use the nested grids, with $n_h = 2^{\ell}+1$ discrete points per dimension where $\ell$ is the refinement level. This leads to $n=2^{\ell+2}$ grid points along the boundary, making a DtN matrix of size $n\times n$. In this DtN matrix, we separate the diagonal and off-diagonal blocks following the strong admissibility condition. It was shown in section~\ref{sec:H_matrix} that these off-diagonal blocks are of low rank, and the matrix completion algorithm could potentially bring benefit if the decoherent and the delocalization conditions are satisfied. We choose two representative square blocks to verify these conditions. They are the block a) and b) demonstrated in Figure~\ref{fig:impedance_2}. As $\ell$ increases, these blocks have larger sizes accordingly: $n_a = n_h-1$, and $n_b = (n_h -1)/2$. 

\begin{figure}[htbp]
\centering
{%
\subfigure[Shepp-Logan phantom]{%
\label{fig:impedance_1}
\includegraphics[trim = 12mm 11mm 16mm 11mm, clip, height = 0.45\textwidth]{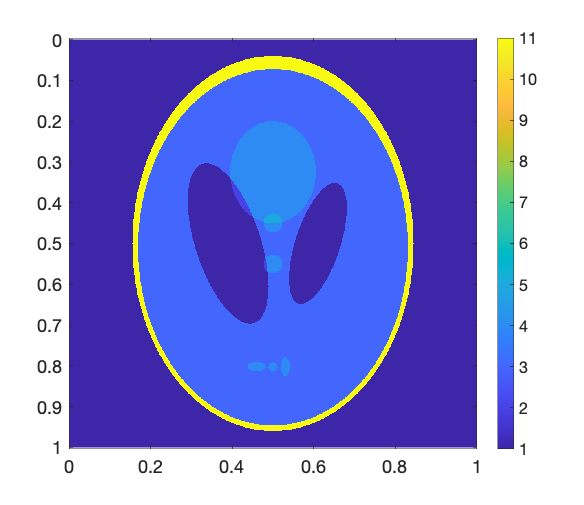}
} 
\subfigure[matrix partition]{%
\label{fig:impedance_2}
\includegraphics[height = 0.45\textwidth]{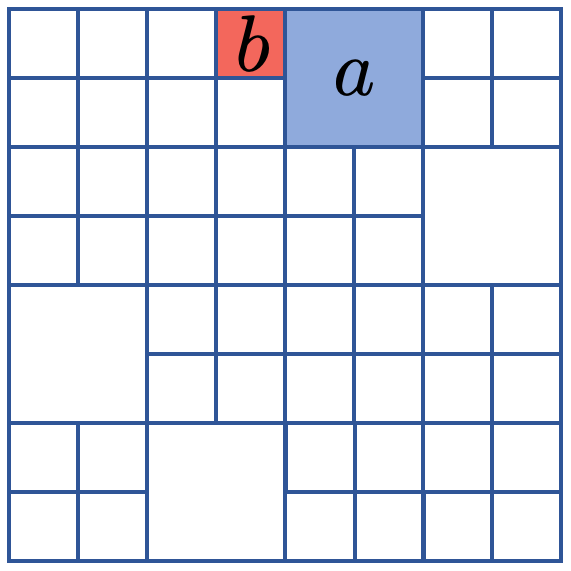}
}
 \caption{Figure \ref{fig:impedance_1} shows the impedance used for the experiments corresponding to the well-known Shepp-Logan phantom, where the color encodes the value of the impedance at each point. Figure \ref{fig:impedance_1} presents a partition of the matrix together with two blocks used for the numerical experiments.}
\label{fig:impedance}
}
\end{figure}

To show the low rank structure of the DtN matrix, we plot in Figure~\ref{fig:partitioning} different levels of $\mathcal{H}$-matrix partitioning. At each level of partitioning, we also plot the approximate rank of each block. The rank is evaluated as the number of singular values above $\epsilon = 10^{-6}$. As can be seen uniformly across all refinement levels, the approximate rank of all off-diagonal blocks is smaller than $5$. We also plot a typical off-diagonal DtN matrix block and its rank structure, shown in Figure~\ref{fig:DtN}. It is clear that these blocks are of low rank.

\begin{figure}[htbp]
{%
\centering
\subfigure[$n_h = 64$]{%
\includegraphics[width=0.48\textwidth]{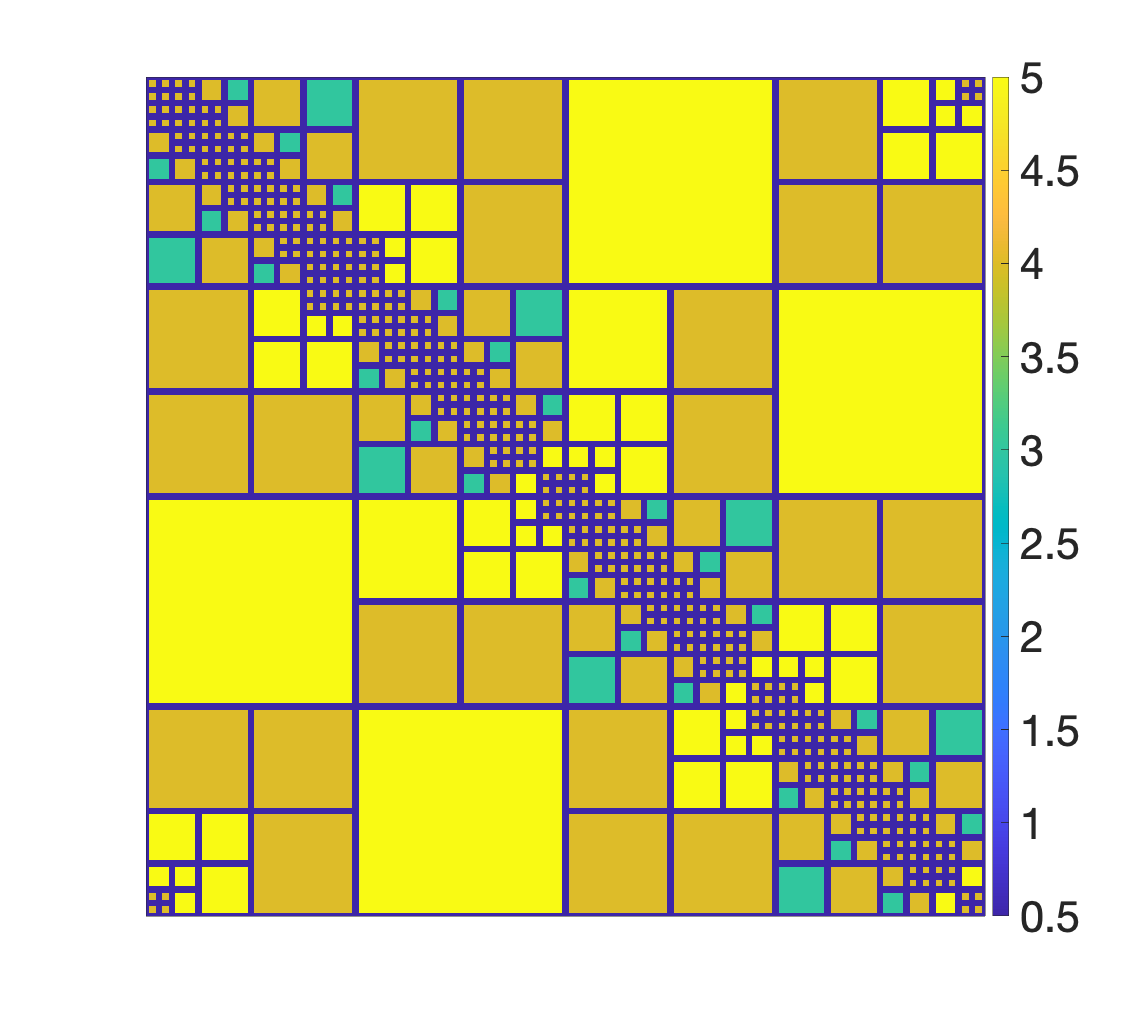}
} 
\subfigure[$n_h = 128$]{%
\includegraphics[width=0.48\textwidth]{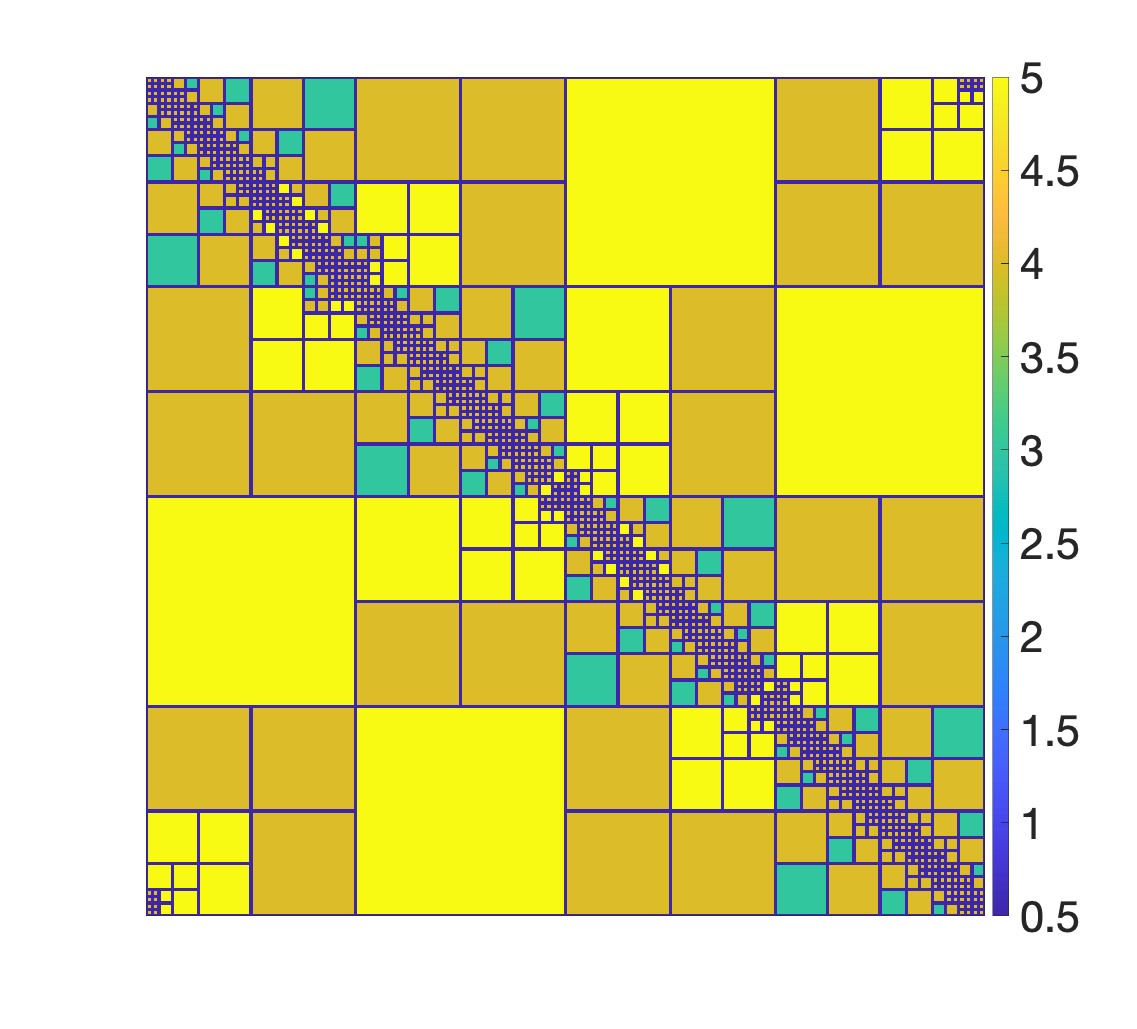}
}\\

\subfigure[$n_h = 256$]{%
\includegraphics[width=0.48\textwidth]{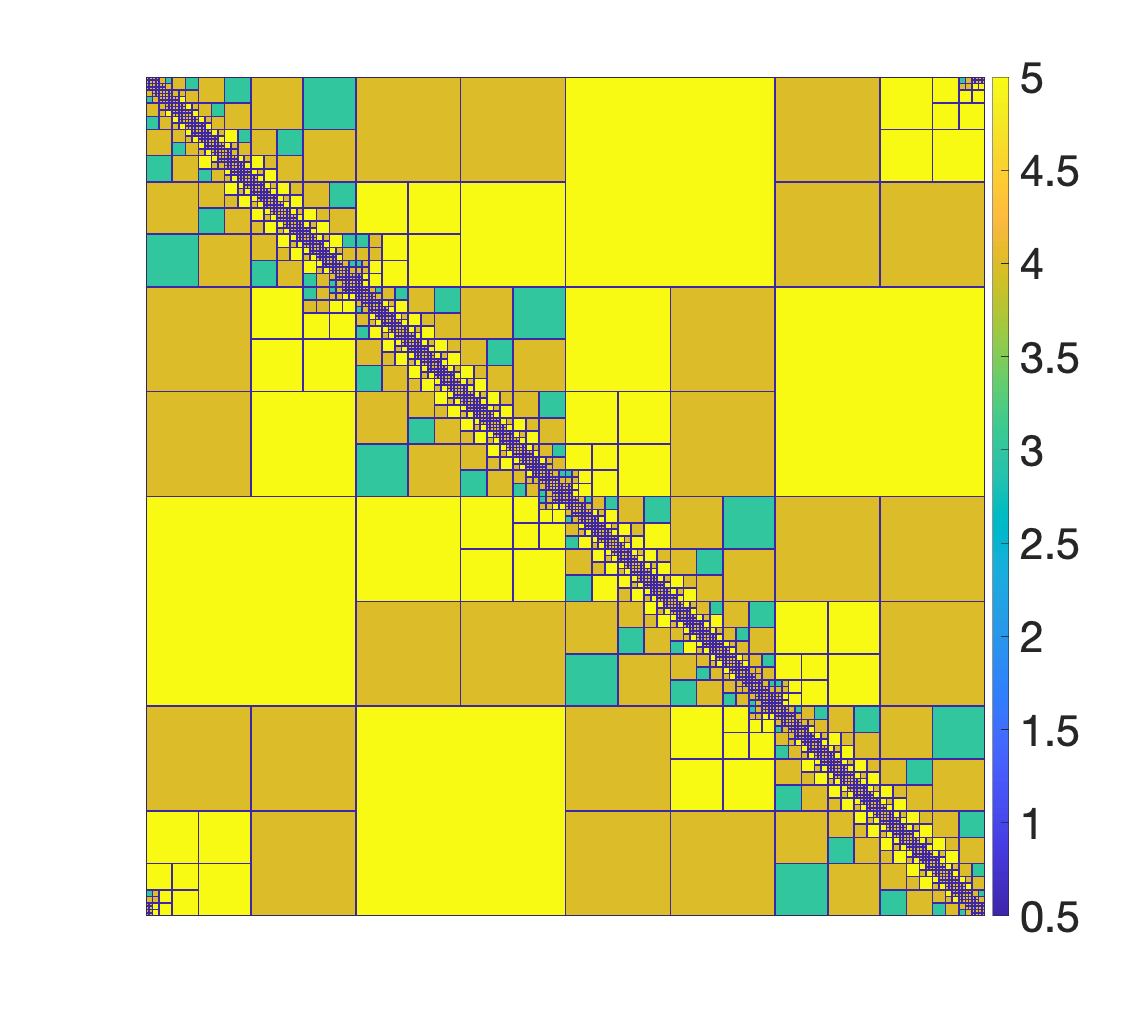}
}
\subfigure[$n_h = 512$]{%
\includegraphics[width=0.48\textwidth]{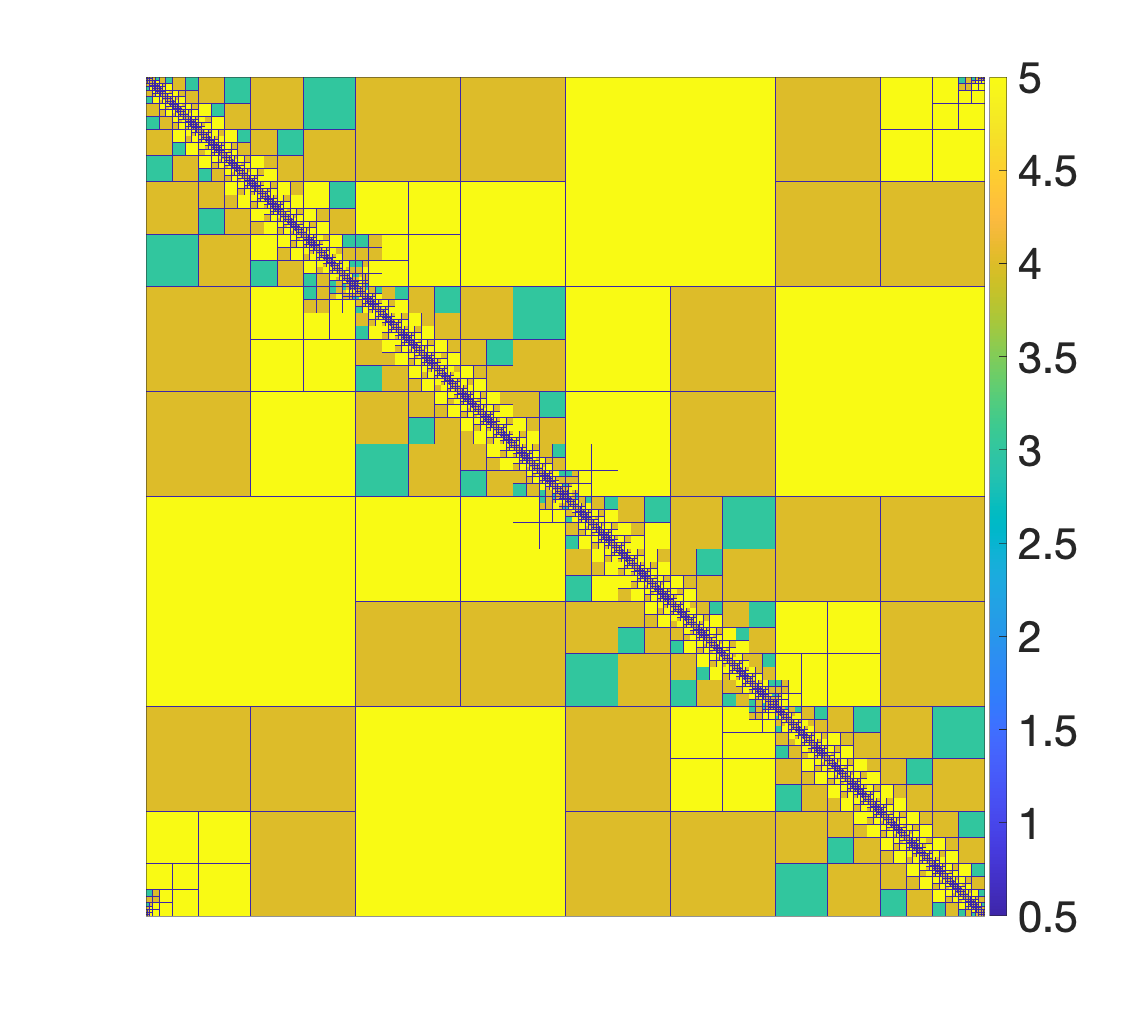}
}
\caption{Partition of the DtN map using different levels of refinements, the blocks are colored with their $\epsilon$-ranks (for $\epsilon = 10^{-6}$).}
\label{fig:partitioning}
}
\end{figure}

\begin{figure}[htbp]
\centering
{%
\subfigure[block of DtN map]{%
\includegraphics[width = 0.45\textwidth]{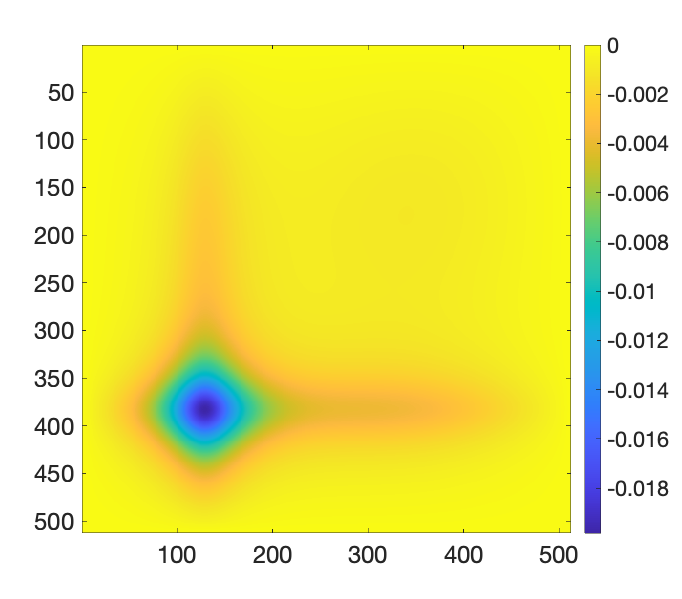}
}
\subfigure[eigenvalues]{%
\includegraphics[width = 0.45\textwidth]{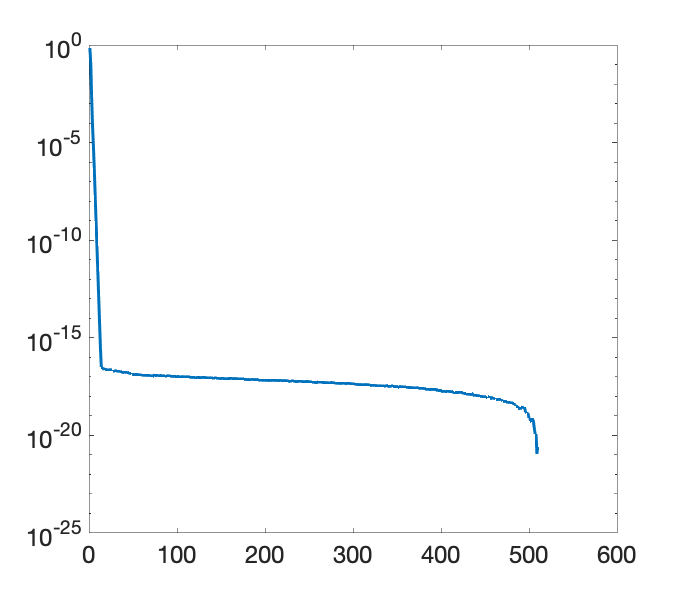}
}
\caption{ Plot of a typical off-diagonal block of the DtN map, where the color correspond to the value of each entry, along with its eigenvalues.}\label{fig:DtN}
}
\end{figure}

To verify the conditions, including the decoherent and delocalization conditions, we plot the coherence indices and the maximum absolute values defined in~\eqref{eqn:def_coherence} for block a) and b). These are shown in Figure~\ref{fig:coherence}. It is clear that as $\ell$ increases, the coherence index stays stable for block a), and only increases slightly for block b), saturating at a relatively small number quickly. The maximum value in the matrix entry evaluation decreases quickly for both matrix blocks, shown in Figure~\ref{fig:max_value}. These evidence suggest that employing the matrix completion algorithm on off-diagonal blocks will provide satisfying results.

\begin{figure}[htbp]
\centering
{%
\subfigure[coherence]{%
\label{fig:coherence}
\includegraphics[trim = 00mm 0mm 12mm 0mm, clip, width = 0.48\textwidth]{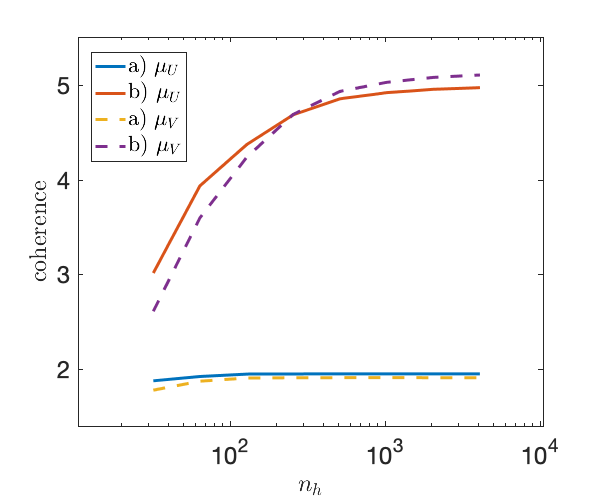}
} 
\subfigure[max val $\mathsf{U}\cdot \mathsf{V}^\ast$]{%
\label{fig:max_value}
\includegraphics[trim = 00mm 0mm 12mm 0mm, clip, width = 0.48\textwidth]{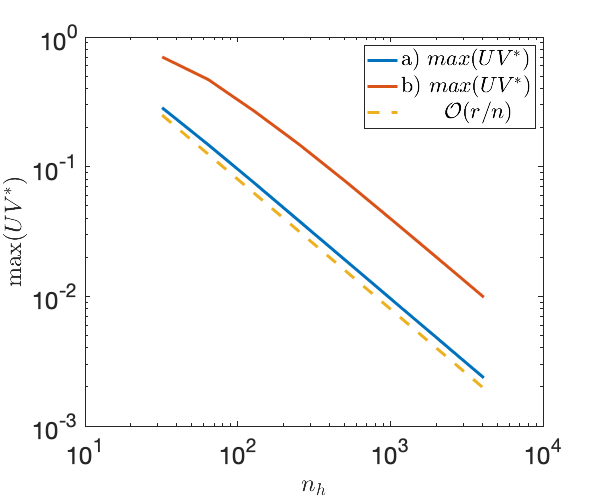}
}
 \caption{ Coherence index and maximum absolute value of $\mathsf{U}\cdot \mathsf{V}^\ast$, for  blocks a) and b) at different levels of refinement in the discretization.}
\label{fig:max_value_all}
}
\end{figure}

Finally we reconstruct the DtN matrix according to Algorithm~\ref{alg:reconstructh}. Since the reconstruction is performed for each block separately, we take the reconstruction of block a) as an example. At each level of the refinement, we select entries from block a) according to the Bernoulli distribution with parameter $p$. We then take the values of these entries as the given data to solve the matrix completion optimization problem \eqref{eqn:opt_nuclear}. In Figure~\ref{fig:reconstruction} we plot the original block, the location of the selected entries, and the reconstruction. Clearly, with $p=0.1$, only ten percent of the data given, we already construct this block with high accuracy.

\begin{figure}[htbp]
{%
\subfigure[original block]{%
\label{fig:reconst_exact}
\includegraphics[trim={10mm 10mm 10mm 10mm}, width = 0.3\textwidth]{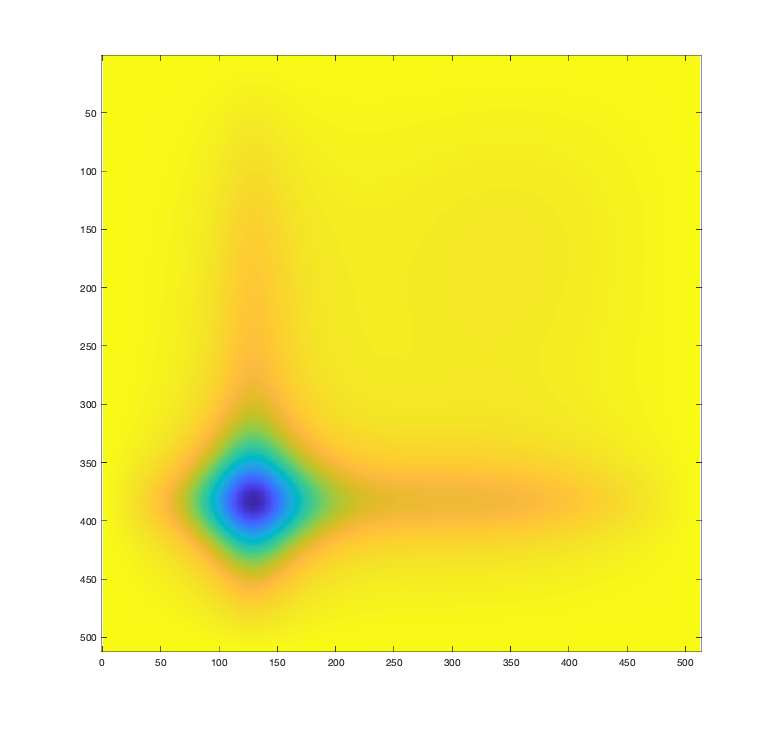}
} 
\subfigure[mask]{%
\label{fig:reconst_spy}
\includegraphics[trim={10mm 10mm 10mm 10mm}, width = 0.3\textwidth]{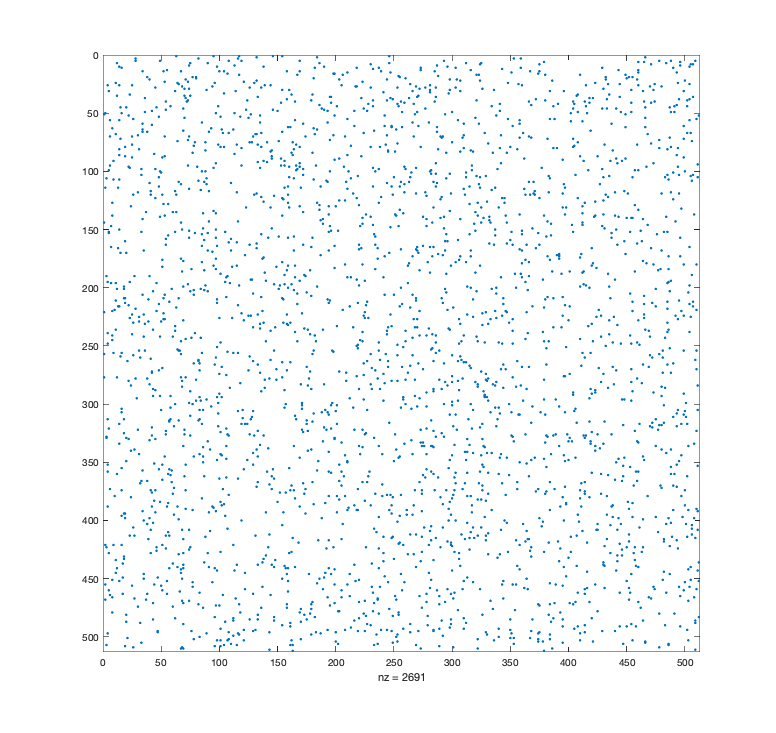}
}
\subfigure[reconstruction]{%
\label{fig:reconst_numer}
\includegraphics[trim={10mm 10mm 10mm 10mm}, width = 0.3\textwidth]{off-diagonal_reconstruction.png}
}
\caption{Reconstruction of block a) in the DtN map with 
with $p = 0.1$.}
\label{fig:reconstruction}
}
\end{figure}

To quantitatively evaluate the algorithm, for each predetermined $p$ and refinement level $\ell$, we perform the selection and reconstruction process $50$ times, and document the success ratio as a function of $p$ and $\ell$. A successful run is defined as a run where the reconstructed block is within $10^{-4}$ error of the ground-truth in Frobenius norm. In Figure~\ref{fig:succes_ratio} we plot the success ratio of reconstructing block a). For low refinement level with coarse discretization, the DtN blocks has small sizes, and the matrix completion algorithm requires a higher percentage of known data for a high success probability of the reconstruction. On refined meshes, small $p$ is sufficient for an accurate reconstruction with high probability. For example, for a matrix of size $512 \time 512$, only up to $5\%$ of the entries are needed to reconstruct the block with a high probability.

\begin{figure}
    \centering
    \includegraphics[height = 0.4\textwidth,trim = 20mm 0mm 20mm 0mm, clip]{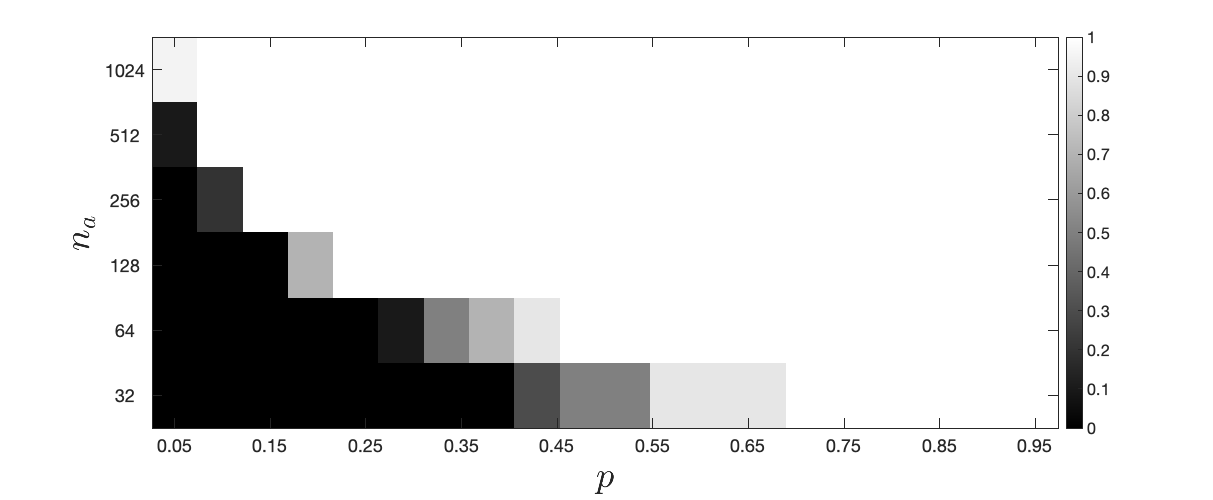}
    \caption{Success ratio for the reconstruction of block a) at different refinements and different density of the sampling mask where $p \sim m/n_a^2$ and $m$ is the number of sampling points. In the plot the color encodes the success ratio, with a lighter color indicating a higher success ratio.}
    \label{fig:succes_ratio}
\end{figure}

\subsection{Reconstruction of the media}\label{sec:media_numerics}

To solve the minimization problem \eqnref{eqn:opt_a}, we use the unconstrained Gauss-Newton method with a constant initial guess. We declare that the optimization algorithm converges if  the gradient norm is less than $10^{-9}$ or the number of iterations exceeds $10,000$. In Figure~\ref{fig:impedance_exact} we plot the groundtruth impedance $\a$, and in Figure~\ref{fig:impedance_reconstructed} we plot the reconstructed media with the exact DtN matrix $\DtNha{\a}$. 
As can be seen, though the reconstructed media captures the two circular blobs it does not resemble the groundtruth. This is not surprising as the unique reconstruction of the media is guaranteed only with the infinite data limit (i.e. the full DtN map, instead of the DtN matrix, is known), no discretization, and infinite precision computation. Improved results can be obtained with total variation regularization, for example, to capture sharp edges of the blobs, but this is beyond the scope of this paper. 
Nevertheless, this reconstruction with the exact DtN matrix provides a benchmark as it is the best case scenario for computation. 
We now obtain $\eval{\DtNha{\a}}_{\Omega}$ by subsampling the exact DtN matrix using the mask in Figure~\ref{fig:impedance_mask} with its design following the criteria in Theorem~\ref{thm:reconstruction}.
Figure~\ref{fig:completed_impedance} shows a reconstruction using the completed matrix $\DtNhRa{\a}$ obtained from $\eval{\DtNha{\a}}_{\Omega}$.  We observe that the result is visibly identical  with the reconstruction using the exact DtN matrix in Figure~\ref{fig:impedance_reconstructed}. As a comparison, we also reconstruct the media by solving~\eqnref{eqn:opt_a} directly with $\eval{\DtNha{\a}}_{\Omega}$ instead of the completed DtN matrix $\DtNhRa{\a}$. As can be seen in Figure~\ref{fig:non_completed_impedance}, the reconstruction using $\eval{\DtNha{\a}}_{\Omega}$ is not able to capture the two blobs.


\begin{figure}[htbp]
{%

\centering
 \includegraphics[width = 0.48\textwidth,trim = 50mm 10mm 20mm 10mm, clip]{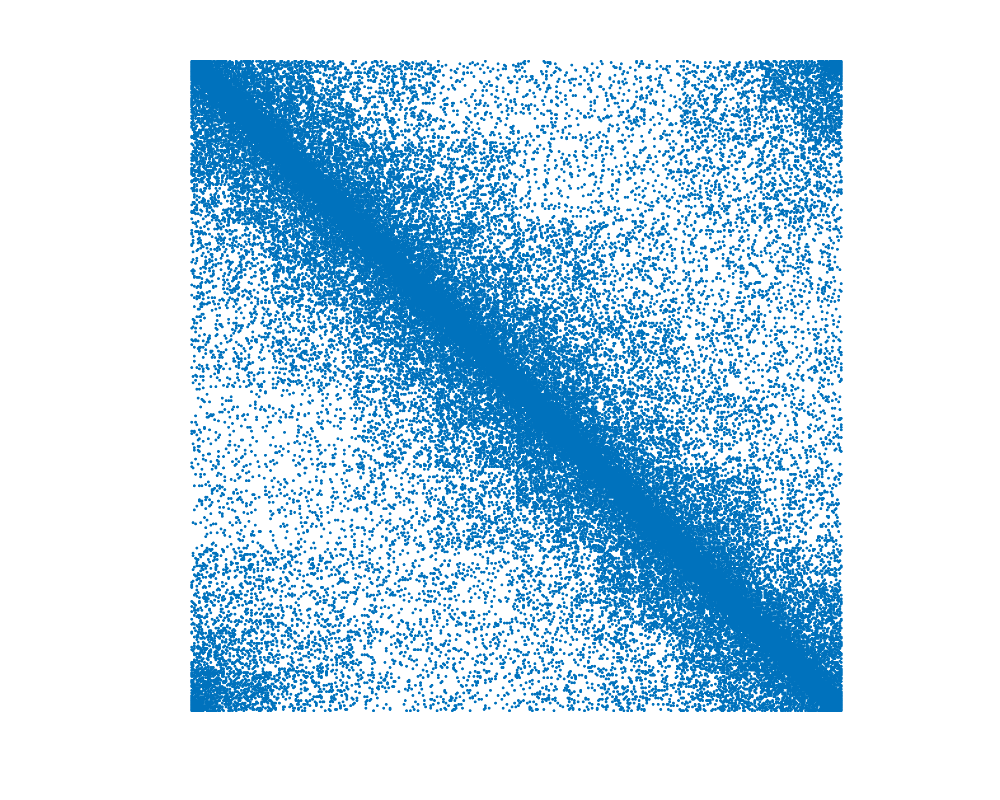}

 \caption{Mask used to down-sample the full (exact) DtN matrix $\DtNha{\a}$. Each blue dot corresponds to a sampled entry, whereas the white color corresponds to not-sampled entries. Note that the number of sampling points is denser close to the diagonal.}\label{fig:impedance_mask}
}
\end{figure}

\begin{figure}[htbp]
{%

\subfigure[exact impedance 
]{%
\label{fig:impedance_exact}
\includegraphics[height = 0.45\textwidth,trim = 50mm 10mm 20mm 10mm, clip]{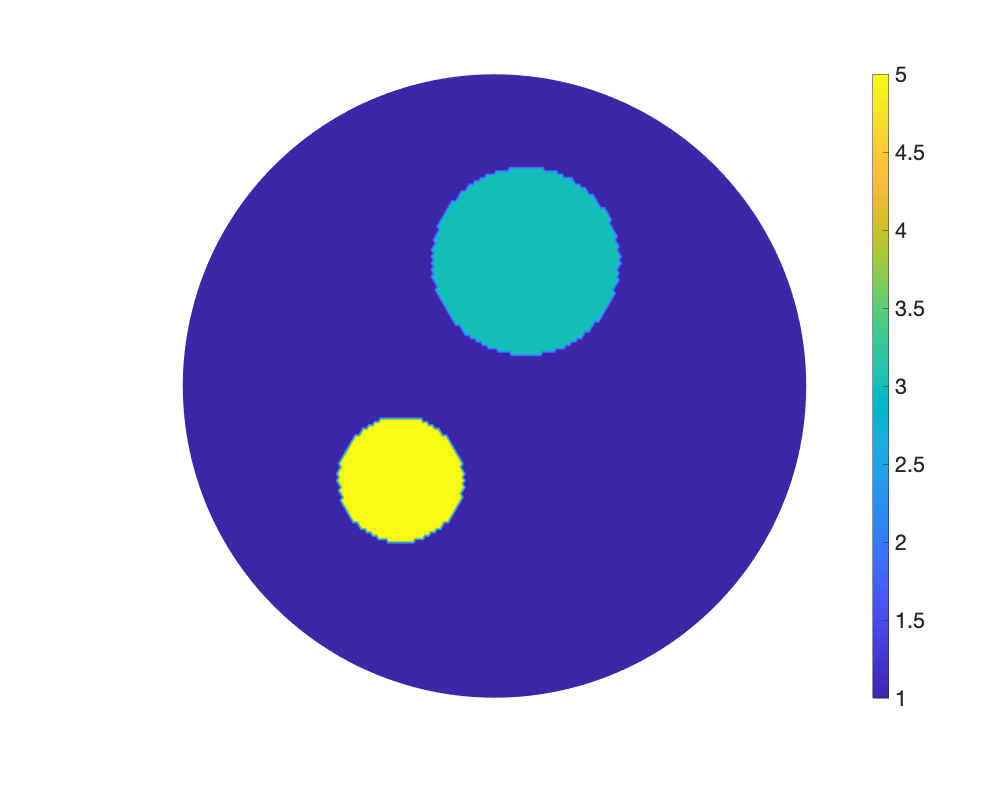}
} 
\subfigure[Reconstruction with $\DtNha{\a}$]{%
\label{fig:impedance_reconstructed}
 \includegraphics[height = 0.45\textwidth,trim = 50mm 10mm 20mm 10mm, clip]{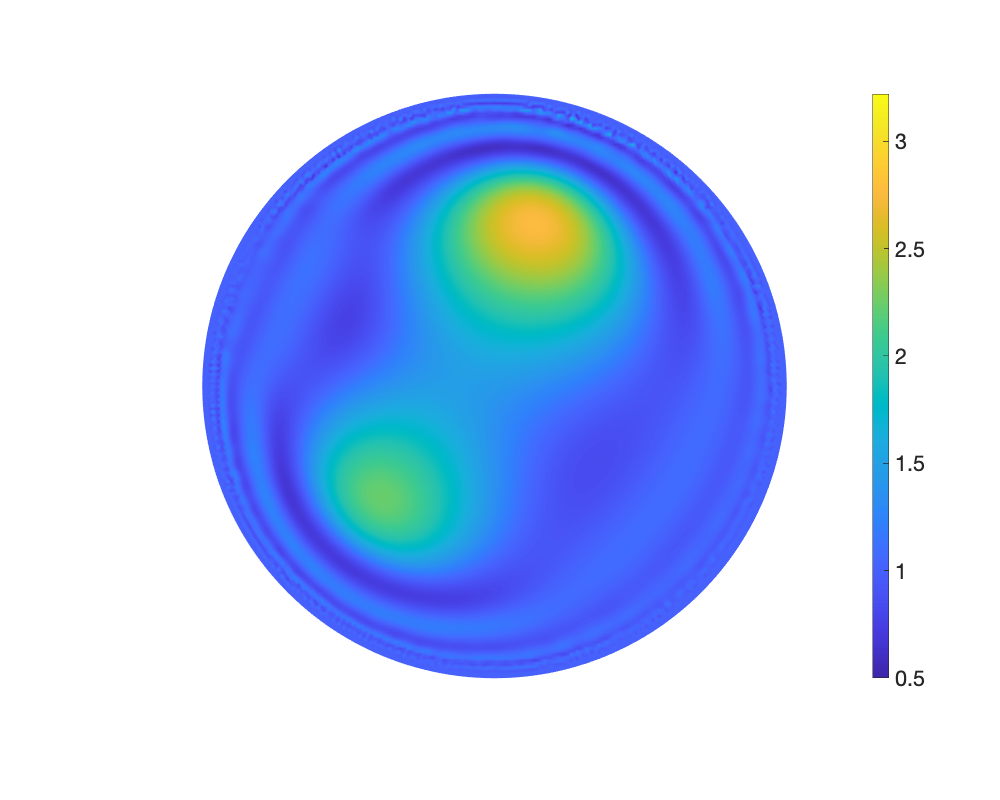}
}
\subfigure[Reconstruction with $\DtNhRa{\a}$]{%
\label{fig:completed_impedance}
\includegraphics[height = 0.45\textwidth,trim = 50mm 10mm 20mm 10mm, clip]{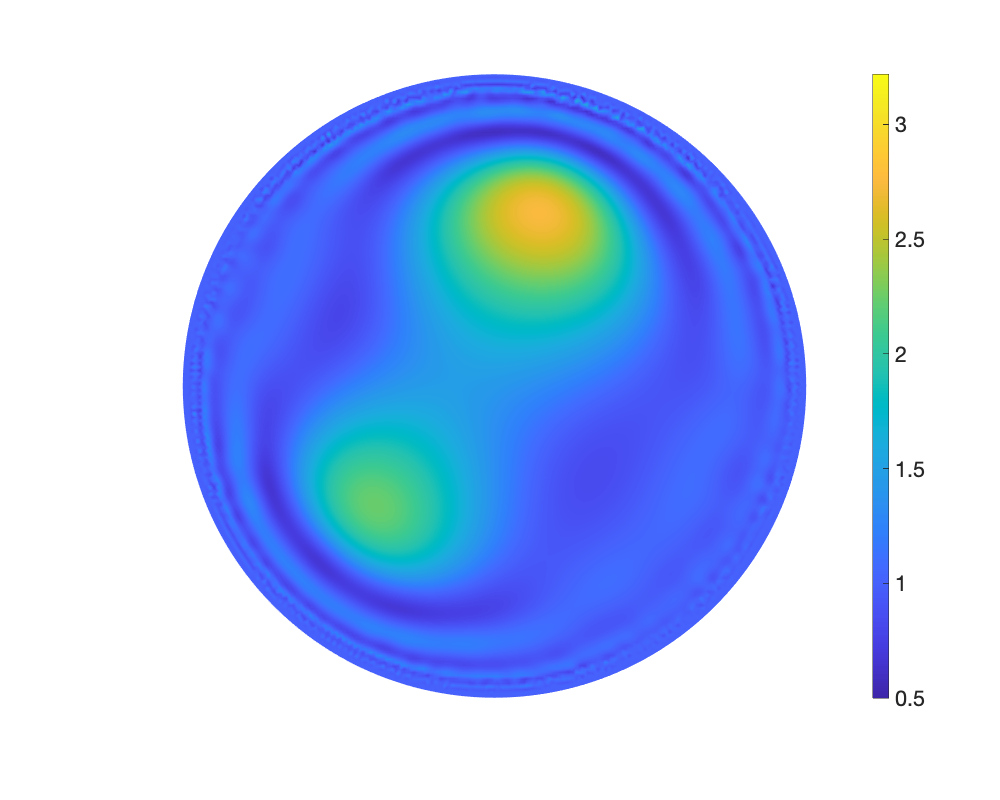}
} 
\subfigure[Reconstruction with $\eval{\DtNha{\a}}_{\Omega}$]{%
\label{fig:non_completed_impedance}
\includegraphics[height = 0.45\textwidth,trim = 50mm 10mm 20mm 10mm, clip]{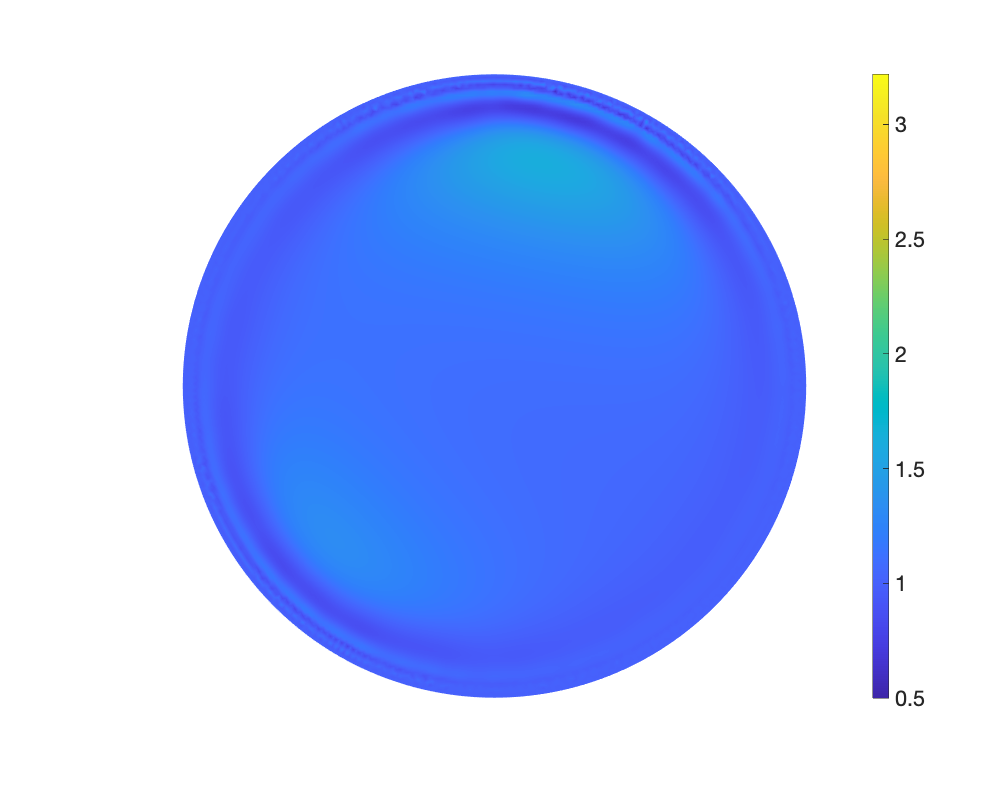}
}
\caption{Reconstructed impedance with the exact DtN matrix $\DtNha{\a}$, with the completed DtN matrix $\DtNhRa{\a}$ reconstructed from a subsampled DtN matrix $\eval{\DtNha{\a}}_{\Omega}$, and directly with the subsampled DtN matrix.}
}
\end{figure}

\subsection{Optical tomography}\label{sec:OT_numerics}
We have used the EIT problem to show that data completion can not only bridge the gap between theoretical and computational inverse problems but also help improve computational inverse solutions. While the former depends on the available theories of the inverse problem under consideration, the latter is expected to be valid for all problems. To demonstrate that data completion is also possible for other problems with $\mathcal{H}$-matrix structure, we now consider an optical tomography problem, where the radiative transfer equation (RTE) serves as the forward model, and its scattering coefficient\textemdash the unknown parameter\textemdash reflects the optical property of the media. More specifically, let $f(x,v)$ presents the density of photon particles at location $x$ moving in direction $v$, then RTE characterizes the dynamics of this distribution function, and in steady state it reads
\begin{equation*}
    v\nabla f = \frac{1}{\Kn}\sigma_s(x)\left[\int_{v'} f\rd{v'} - f\right]\,.
\end{equation*}
Here the left hand side describes the particles moving in direction $x$ with velocity $v$, and the term on the right suggests the scattering with the intensity characterized by $\sigma_s$. $\Kn$ is called the Knudsen number. 
The ``inflow" part of the boundary
\[
\Gamma_-=\{(x,v): x\in\partial\mathcal{D}\,,v\cdot n_x <0\}\,
\]
is where lights are shined into the media, and one takes measurement on the ``outflow" part of the boundary
\[
\Gamma_+=\{(x,v): x\in\partial\mathcal{D}\,,v\cdot n_x >0\}\,.
\]
The map that directs incoming data to the outgoing data is known as the albedo operator and is used to reconstruct $\sigma_s$.

 Figure~\ref{fig:rte_compression_diffusive_1} plots the albedo matrix (discretized albedo operator), along with its eigenvalues in Figure~\ref{fig:rte_compression_diffusive_2} in the diffusion regime $\Kn\ll 1$. It is clear that the operator is approximately low rank. In this case, one would be able to approximate the full operator by solving the optimization problem in \eqref{eqn:opt_nuclear}.

In the ballistic regime, $\Kn\sim 1$, then the albedo matrix, plotted in Figure~\ref{fig:rte_compression_1}, is no longer of low rank. However, it is approximately an $\mathcal{H}$-matrix. A partition of the albedo matrix is shown in Figure~\ref{fig:rte_compression_2}, and we plot the $\epsilon$-rank (with $\epsilon =10^{-6}$) for all the blocks in Figure~\ref{fig:rte_compression_3}. As can be seen, the $\epsilon$-rank is uniformly bounded by $5$ in each block.
Analogous to the DtN matrix, we present the reconstruction of one typical off-diagonal block in this albedo matrix. For the block shown in Figure~\ref{fig:rte_compression_2}, we select data according to the Bernoulli distribution with parameter $p$, and the selected entries serve as given data in the matrix completion algorithm.  Figure~\ref{fig:succes_ratio_RTE} plots the success ratio, computed with $20$ experiments for each $p$ and refinement level. Here success means the reconstructed matrix is within $10^{-4}$ error in Frobenius norm of the groundtruth. It is clear that the chance of successful reconstruction increases as the dimension of the matrix increases, as predicted by the theory. We leave the detailed  bridging-the-gap analysis and  parameter reconstructions for future work.

\begin{figure}[htbp]
{%
\subfigure[Albedo matrix]{%
\label{fig:rte_compression_diffusive_1}
\includegraphics[trim={10mm 12mm 0mm 10mm}, clip, width = 0.45\textwidth]{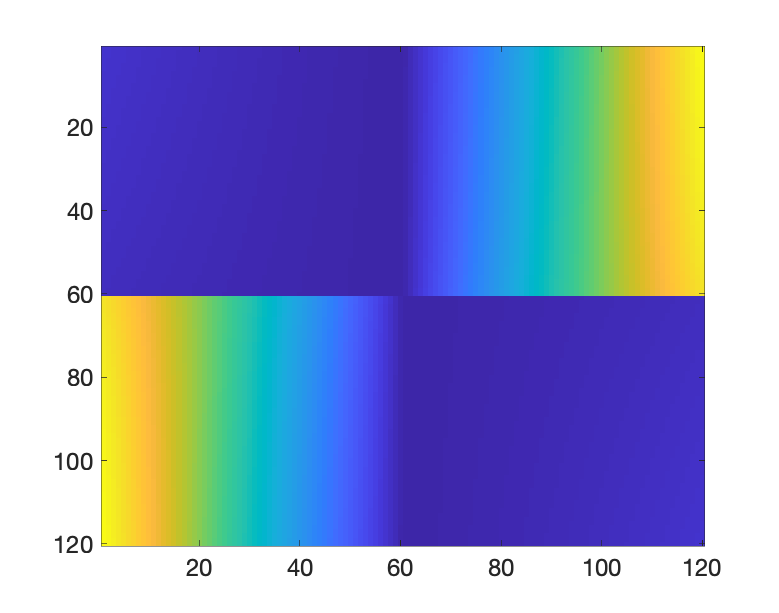}
} 
\subfigure[Eigenvalues of the Albedo matrix]{%
\label{fig:rte_compression_diffusive_2}
\includegraphics[trim={10mm 12mm 0mm 10mm}, clip, width = 0.45\textwidth]{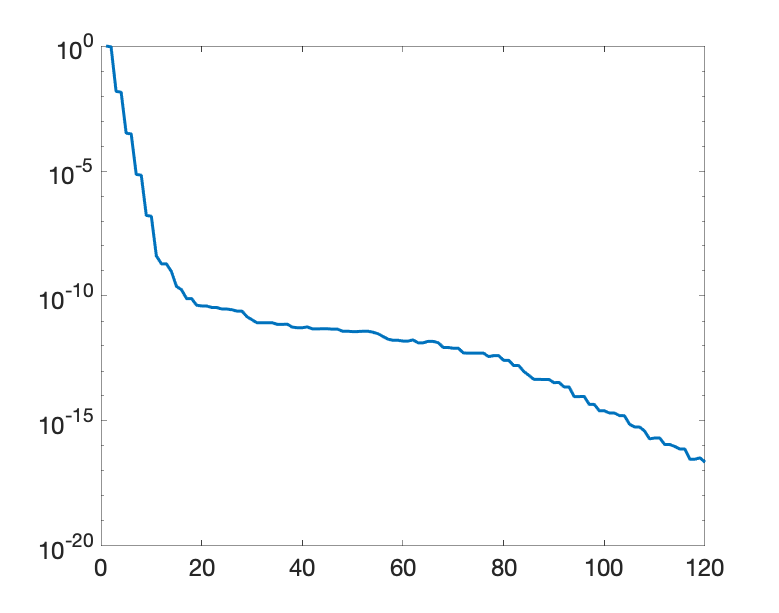}
}
\caption{Figure~\ref{fig:rte_compression_diffusive_1} colors the albedo matrix in the diffusive regime when $\Kn = 2^{-5}$ with the values of its entries. Figure~\ref{fig:rte_compression_diffusive_2} plots eigenvalues of the albedo matrix.}
\label{fig:rte_compression_diffusive}
}
\end{figure}

\begin{figure}[htbp]
{%
\subfigure[albedo operator]{%
\label{fig:rte_compression_1}
\includegraphics[trim={10mm 10mm 0mm 10mm}, height = 0.28\textwidth]{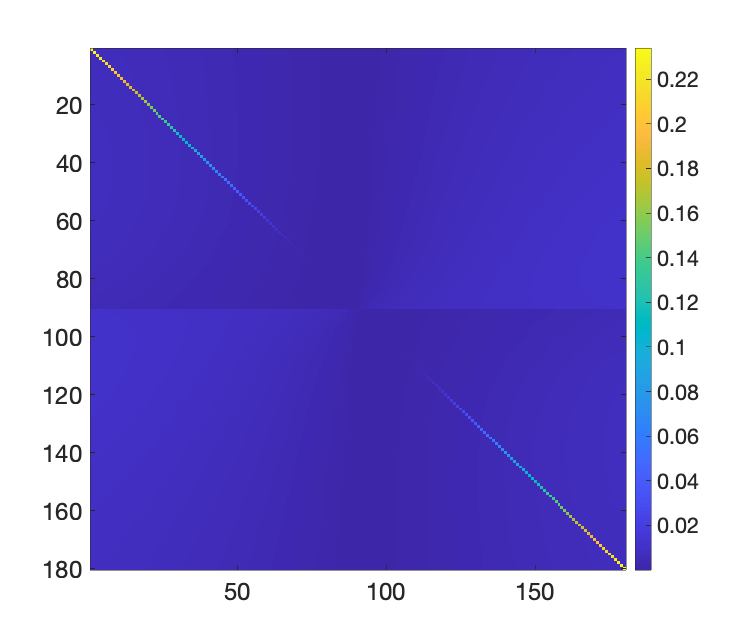}
} 
\subfigure[partitioning]{%
\label{fig:rte_compression_2}
\includegraphics[trim={0mm 0mm 0mm 0mm}, clip, height = 0.27\textwidth]{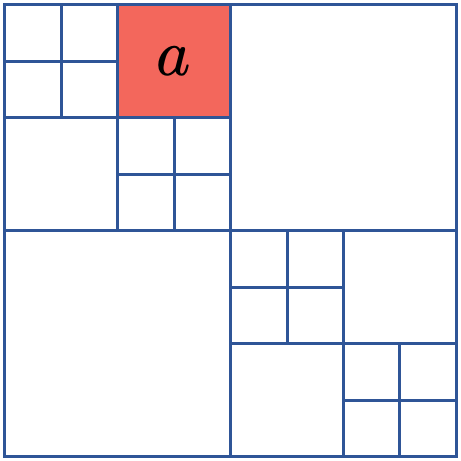}
}
\subfigure[ranks]{%
\label{fig:rte_compression_3}
\includegraphics[trim={10mm 10mm 0mm 10mm}, height = 0.28\textwidth]{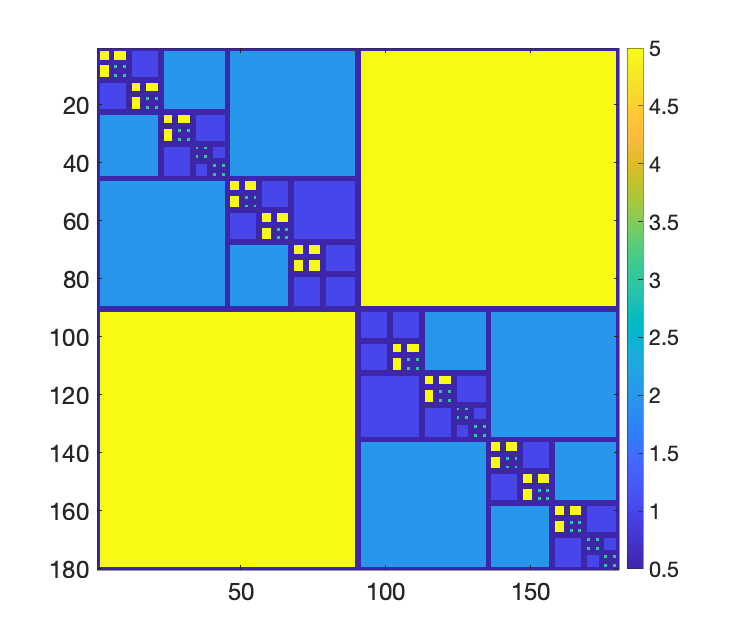}
}
\caption{From left to right, plot of the albedo matrix for $\Kn = 1$ , partitioning of the albedo matrix, and the $\epsilon$-rank of each block in the partition. Here each block is colored with its $\epsilon$-rank.}
\label{fig:rte_compression}
}
\end{figure}

\begin{figure}
    \centering
    \includegraphics[height = 0.4\textwidth,trim = 20mm 0mm 20mm 0mm,clip]{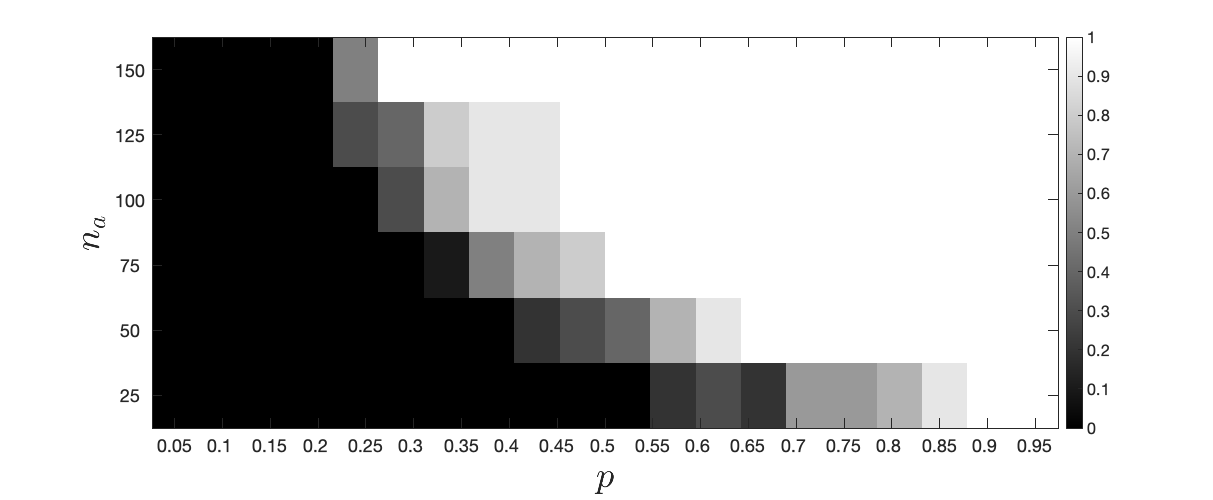}
    \caption{ Success ratio for the reconstruction of block a) in Figure~\ref{fig:rte_compression_2} of the albedo matrix with $\Kn = 1$ at different refinement level and different value of $p$. In the plot the color encodes the success ratio.}
    \label{fig:succes_ratio_RTE}
\end{figure}

\section{Conclusions}\label{sec:conclusions}
There is a gap between theoretical and numerical approaches for inverse problems. While in theory, infinite-dimensional datasets, encoded in the ItO map, are available to infer a function living in infinite dimensional space, in the numerical and experimental settings, both the available data and the reconstructed parameter are finite-dimensional. This mismatch prevents the application of the theory in guiding and improving practical computational inverse solutions: In very rare cases can one assert the unique and stable reconstruction of discretized parameters.

We have presented a framework to bridge this gap using data completion with the EIT problem as the testbed. In particular, we view finitely experimental data points as entries in the input-to-output\textemdash Dirichlet-to-Neumann (DtN)\textemdash matrix. Since the DtN matrix has the $\mathcal{H}$-matrix structure,  we exploit the off-diagonally low-rank property and the matrix completion technique to informatively collect mostly random data points in the matrix, and fill in the unknown entries with a matrix completion method.
The goal of matrix completion is twofold: \texttt{I)}{\em bridging the gap} and \texttt{II)} {\em improving the quality of computational inverse solutions}. 
For \texttt{I)}, the DtN map is rigorously recovered by lifting the completed DtN matrix, up to discretization error with high probability.  
This allows us to apply the inverse theory to asymptotically show the unique and stable reconstruction of parameters.
For \texttt{II)} we have numerically demonstrated that\textemdash unlike traditional computational inverse problems that uses the incomplete DtN matrix\textemdash we deploy the completed DtN matrix to reconstruct the unknown parameters. 
The numerical results have shown that the reconstructions using the completed DtN matrix and the exact DtN matrix are visibly identical, while the reconstruction directly from incomplete DtN matrix is completely off.


We emphasize that the goal of the current paper is to propose a general framework to bridge and improve theoretical and computational inverse problems. For a thorough error analysis, we need a more precise estimate of the decay of the singular values in each block of the input-to-output matrix. This highly depends on the specific equation encoded in the forward map. This part of error analysis is not yet available in its most precise form in the literature, and thus is left for future work.

\newpage

\begin{appendix}
\section{Proof of Theorem~\theoref{asymptoticUniquenessFull}}\label{sec:appendix_proof}
Let us define $\Phit$ as the unique solution of the following problem
\begin{equation}
\eqnlab{ellipticWeakt}
\int_{\Omega}\a\,\Grad \Phit \cdot \Grad \v\,d\Omega = 0, \quad \eval{\Phit}_\pOmega = \phih, \quad \forall \v \in \Hone_0\LRp{\Omega},
\end{equation}
where again $\phih = \Pih\phi$. Let us denote $\DtNa{\a}^\dagger$ via
\begin{equation}
\eqnlab{DtNat}
\LRa{\DtNa{\a}^\dagger\phi,\psi}:=
\int_{\Omega}\a\,\Grad \Phit \cdot \Grad \Psi\,d\Omega,
\end{equation}
where $\Psi \in H^1\LRp{\Omega}$ can be any extension of $\psi$ such that $\eval{\Psi}_{\pOmega} = \psi$.
\begin{lemma}
\lemlab{semiError}
There holds
\begin{equation}
\eqnlab{semiError}
\nor{\DtNa{a} - \DtNa{\a}^\dagger}_{\H^{1/2}\LRp{\pOmega}\to\H^{-1/2}\LRp{\pOmega}} \le c\nor{\I - \Pih}_{\H^{1/2}\LRp{\pOmega}\to\H^{1/2}\LRp{\pOmega}},    
\end{equation}
 where $\I$ is the identity map and $c$ is a constant independent of the mesh size $\h$. 
\end{lemma}
\begin{proof}
By definition we have
\begin{align*}
\snor{\LRa{\DtNa{a}\phi,\psi}- \LRa{\DtNa{\a}^\dagger\phi,\psi}} &= \snor{\int_{\Omega}\a\,\Grad \LRp{\Phi - \Phit} \cdot \Grad \Psi\,d\Omega} \\ &\le c \nor{\psi}_{\H^{1/2}\LRp{\pOmega}}\nor{\Grad \LRp{\Phi - \Phit}}_{\L^2\LRp{\Omega}} \\
&\le c \nor{\psi}_{\H^{1/2}\LRp{\pOmega}} \nor{\phi}_{\H^{1/2}\LRp{\pOmega}} \nor{\I - \Pih}_{\H^{1/2}\LRp{\pOmega}\to\H^{1/2}\LRp{\pOmega}},
\end{align*}
where we have used the uniform boundedness of $\a$, and definition~\eqnref{ellipticWeakt}. The estimate \eqnref{semiError} thus follows.
\end{proof}

Let $\P: \Hone\LRp{\Omega} \ni \Phit \mapsto \P\Phit \in \Vh_{\phih}\LRp{\Omega}$,  where $\Vh_{\phih} := \LRc{\v \in \Vh\LRp{\Omega}: \eval{v}_{\pOmega} = \phih}$,  be defined as
\[
\int_{\Omega}a\,\Grad \P\Phit \cdot \Grad \vh\,d\Omega = \int_{\Omega}a\,\Grad \Phit \cdot \Grad \vh\,d\Omega, 
\quad \forall \vh \in \VhO.
\]
Note that $\P$ is a well-defined linear bounded map and $\Phih = \P\Phit$, where $\Phih$ is the FEM solution.
\begin{lemma}
\lemlab{discreteError}
There holds
\begin{multline}
\nor{\DtNa{\a}^\dagger-\DtNRhat{\a}}_{\H^{1/2}\LRp{\pOmega}\to\H^{-1/2}\LRp{\pOmega}} \le c\nor{\I - \P}_{\Hone\LRp{\Omega}\to \Hone\LRp{\Omega}} \nonumber\\
+ c\nor{\I - \Pih}_{\H^{1/2}\LRp{\pOmega}\to\H^{1/2}\LRp{\pOmega}},    
\eqnlab{discreteError}
\end{multline}
where $\I$ is the identity map and $c$ is a constant independent of the meshsize $\h$. 
\end{lemma}
\begin{proof}
We have
\begin{align*}
&\snor{\LRa{\DtNa{\a}^\dagger-\DtNRhat{\a}\phi,\psi}} \\
\le&\snor{\int_{\Omega}\a\,\Grad \LRp{\Phit - \P\Phit} \cdot \Grad \Psi\,d\Omega} + \snor{\int_{\Omega}\a\,\Grad \P\Phit \cdot \Grad \LRp{\Psi - \Psi^\h}\,d\Omega}\\
\le& c \nor{\psi}_{\H^{1/2}\LRp{\pOmega}} \nor{\phi}_{\H^{1/2}\LRp{\pOmega}}\nor{\I - \P}_{\Hone\LRp{\Omega}\to \Hone\LRp{\Omega}} \\
 &    + c \nor{\phi}_{\H^{1/2}\LRp{\pOmega}}\LRp{\nor{\Psi - \tilde{\Psi}}_{\Hone\LRp{\Omega}}+ \nor{\tilde{\Psi} - \P\tilde{\Psi}}_{\Hone\LRp{\Omega}}} \\
 \le& c\nor{\psi}_{\H^{1/2}\LRp{\pOmega}} \nor{\phi}_{\H^{1/2}\LRp{\pOmega}}\LRp{\nor{\I - \P}_{\Hone\LRp{\Omega}\to \Hone\LRp{\Omega}} + \nor{\I - \Pih}_{\H^{1/2}\LRp{\pOmega}\to\H^{1/2}\LRp{\pOmega}}},
\end{align*}
where we have defined $\tilde{\Psi}$ as the solution \eqnref{ellipticWeakt} with boundary data $\Pih\psi$ and taken $\Psi^h = \P\tilde{\Psi}$.
\end{proof}

\begin{proof}[Proof of Theorem \theoref{asymptoticUniquenessFull}]
We provide the proof of the first assertion as the others are obvious owing to \eqref{eqn:well_cont} and the triangle inequality. From Lemmas \lemref{semiError}--\lemref{discreteError} and the triangle inequality we need to show that
\[
\lim_{\h\to 0} \nor{\I - \P}_{\Hone\LRp{\Omega}\to \Hone\LRp{\Omega}} = 0, \quad \text{ and }
 \lim_{\h\to 0}\nor{\I - \Pih}_{\H^{1/2}\LRp{\pOmega}\to\H^{1/2}\LRp{\pOmega}} = 0.
\]
It is sufficient to prove the former as the proof for the latter is similar. By definition we have
\begin{multline*}
\nor{\I - \P}_{\Hone\LRp{\Omega}\to \Hone\LRp{\Omega}} = 
\sup_{\nor{\Psi}_{\Hone\LRp{\Omega} \le 1}} \sup_{\nor{\Phi}_{\Hone\LRp{\Omega} \le 1}}\LRa{\LRp{\I-\P}\Psi,\Phi}_{\Hone\LRp{\Omega}}
\\= \LRa{\LRp{\I-\P}\Psi^*,\Phi^*}_{\Hone\LRp{\Omega}},
\end{multline*}
where $\LRa{\cdot,\cdot}_{\Hone\LRp{\Omega}}$ denotes the inner product in $\Hone\LRp{\Omega}$, and we have used the fact that the suprema are attainable \cite{James57,James64} at some $\Psi^*$ and $\Phi^*$. The density of the finite element space $\Vh\LRp{\Omega}$ in $\Hone\LRp{\Omega}$ as $\h \to 0$ concludes the proof of the first assertion. 
\end{proof}

\end{appendix}
\bibliographystyle{siamplain}
\bibliography{ref,references,Extra}

\end{document}